\documentclass{amsart}
\pdfoutput=1
\usepackage{amsmath, amssymb, amsthm, mathrsfs, tikz,latexsym}
\usepackage{multirow}
\usepackage{hyperref,enumitem,needspace}
\usepackage{comment}
\usetikzlibrary{matrix,arrows,positioning}
\usetikzlibrary{decorations.pathreplacing}
\newcommand{\Sn}{\mathfrak{S}_n}
\newcommand{\RCT}[1]{\ensuremath{\mathrm{RCT}(#1)}}
\newcommand{\C}{\mathscr{C}}

\newcommand{\PW}{\ensuremath{\mathrm{ParseWords}}}
\newcommand{\rk}{\ensuremath{\mathrm{rank}}}
\newcommand{\meet}{\wedge}
\newcommand{\join}{\vee}
\newcommand{\bT}{\mathbb{T}}
\newcommand{\cJ}{\mathcal{J}}

\newenvironment{enumrom}{\begin{enumerate}}{\end{enumerate}}

\renewcommand{\textbf}{\emph}

\theoremstyle{plain}
\newtheorem{thm}{Theorem}[section]
\newtheorem{cor}[thm]{Corollary}
\newtheorem{lemma}[thm]{Lemma}

\newtheorem{prop}[thm]{Proposition}

\theoremstyle{definition}
\newtheorem{defn}[thm]{Definition}
\newtheorem*{example}{Example}

\theoremstyle{remark}
\newtheorem*{note}{Note}
\newtheorem*{remark}{Remark}

\title{On a Subposet of the Tamari Lattice}
\author{Sebastian A.\ Csar}
\address{School of Mathematics, University Of Minnesota, Minneapolis, MN 55455, USA}
\email{\href{mailto:csarx001@math.umn.edu}{csar@math.umn.edu}}
\author{Rik Sengupta}
\address{Department of Mathematics, Massachusetts Institute of Technology, MA 02139, USA}
\email{\href{mailto:rsengupt@mit.edu}{rsengupt@mit.edu}}
\author{Warut Suksompong}
\address{Department of Mathematics, Massachusetts Institute of Technology, MA 02139, USA}
\email{\href{mailto:warutsuk@mit.edu}{warutsuk@mit.edu}}

\begin{document}
\maketitle
\begin{abstract}
	We explore some of the properties of a subposet of the Tamari lattice introduced by Pallo, which we call the comb poset. We show that a number of binary functions that are not well-behaved in the Tamari lattice are remarkably well-behaved within an interval of the comb poset: rotation distance, meets and joins, and the common parse words function for a pair of trees. We relate this poset to a partial order on the symmetric group studied by Edelman.
\end{abstract}

\section{Introduction} \label{Section 1}

The set $\mathbb{T}_n$ of all full binary trees with $n$ leaves, or parenthesizations of $n$ letters, has been well-studied, and carries much structure. Its cardinality $|\mathbb{T}_n|$ is the $(n - 1)$\textsuperscript{th} Catalan number \[C_{n - 1} = \frac{1}{n} \binom{2n - 2}{n - 1}.\] The \emph{rotation graph}, $\mathscr{R}_n$, is the graph with vertex set $\mathbb{T}_n$, in which edges correspond to a local change in the tree called a \textbf{rotation}, corresponding to changing a single parenthesis pair in the parenthesization. This graph $\mathscr{R}_n$ forms the vertices and edges of an $(n - 2)$-dimensional convex polytope called the \textbf{associahedron}, $K_{n+1}$. If we direct the edges of $\mathscr{R}_n$ in a certain fashion, we obtain the Hasse diagram for the well-studied \textbf{Tamari lattice}, $\mathscr{T}_n$, on $\mathbb{T}_n$, shown below for $n = 4$.
\begin{center}
\begin{tikzpicture}
 
    \tikzstyle{every node}=[draw,circle,fill=black,minimum size=4pt,
                            inner sep=0pt]
    
    \draw(0,1) node(A) {};
    \draw(-3,2.5) node(B) {};
    \draw(-3,6) node(C) {};
    \draw(0,7.5) node(D) {};
    \draw(2.75,4.25) node(E) {};
    
    \draw[very thick,blue] (A) -- (B);
    \draw[very thick,blue] (B) -- (C);
    \draw[very thick,blue] (C) -- (D);
    \draw[very thick,blue] (D) -- (E);
    \draw[very thick,blue] (E) -- (A);

    \tikzstyle{every node}=[draw,circle,fill=black,minimum size=1pt,
                            inner sep=0pt]
    
    \draw(0.25,0) node(A1) {};
    \draw(0.75,0) node(A2) {};
    \draw(0,0.25) node(A3) {};
    \draw(0.5,0.25) node(A4) {};
    \draw(-0.25,0.5) node(A5) {};
    \draw(0.25,0.5) node(A6) {};
    \draw(0,0.75) node(A7) {};
    \draw (A1) -- (A4);
    \draw (A2) -- (A4);
    \draw (A3) -- (A6);
    \draw (A4) -- (A6);
    \draw (A5) -- (A7);
    \draw (A6) -- (A7);
    
    \draw(-4,2) node(B1) {};
    \draw(-3.5,2) node(B2) {};
    \draw(-3.75,2.25) node(B3) {};
    \draw(-3.25,2.25) node(B4) {};
    \draw(-4,2.5) node(B5) {};
    \draw(-3.5,2.5) node(B6) {};
    \draw(-3.75,2.75) node(B7) {};
    \draw (B1) -- (B3);
    \draw (B2) -- (B3);
    \draw (B3) -- (B6);
    \draw (B4) -- (B6);
    \draw (B5) -- (B7);
    \draw (B6) -- (B7);
    
    \draw(-4,5.5) node(C1) {};
    \draw(-3.5,5.5) node(C2) {};
    \draw(-4.25,5.75) node(C3) {};
    \draw(-3.75,5.75) node(C4) {};
    \draw(-4,6) node(C5) {};
    \draw(-3.5,6) node(C6) {};
    \draw(-3.75,6.25) node(C7) {};
    \draw (C1) -- (C4);
    \draw (C2) -- (C4);
    \draw (C3) -- (C5);
    \draw (C4) -- (C5);
    \draw (C5) -- (C7);
    \draw (C6) -- (C7);
    
    \draw(-0.5,7.75) node(D1) {};
    \draw(0,7.75) node(D2) {};
    \draw(-0.25,8) node(D3) {};
    \draw(0.25,8) node(D4) {};
    \draw(0,8.25) node(D5) {};
    \draw(0.5,8.25) node(D6) {};
    \draw(0.25,8.5) node(D7) {};
    \draw (D1) -- (D3);
    \draw (D2) -- (D3);
    \draw (D3) -- (D5);
    \draw (D4) -- (D5);
    \draw (D5) -- (D7);
    \draw (D6) -- (D7);
    
    \draw(3,4) node(E1) {};
    \draw(3.5,4) node(E2) {};
    \draw(3.75,4) node(E3) {};
    \draw(4.25,4) node(E4) {};
    \draw(3.25,4.25) node(E5) {};
    \draw(4,4.25) node(E6) {};
    \draw(3.625,4.5) node(E7) {};
    \draw (E1) -- (E5);
    \draw (E2) -- (E5);
    \draw (E3) -- (E6);
    \draw (E4) -- (E6);
    \draw (E5) -- (E7);
    \draw (E6) -- (E7);
    
\end{tikzpicture}
\end{center}
The Tamari lattice has many properties, but it has certain deficiencies. For instance, it is not ranked. Although one can encode the Tamari order by componentwise comparison of \emph{weight vectors} $\langle T \rangle \in \{0, 1, \ldots, n - 2\}^{n - 1}$ for $T \in \mathscr{T}_n$, introduced by Huang and Tamari in \cite{ht:bracketingvectors} for the lattice dual to $\mathscr{C}_n$, only the meet is given by the componentwise minimum of these weight vectors; the join cannot be characterized similarly. Furthermore, computing the \textbf{rotation distance} $d_{\mathscr{R}_n}(T_1, T_2)$ between two trees $T_1, T_2$ in the graph $\mathscr{R}_n$ does not appear to follow easily from knowing their meet and join in the Tamari lattice.

Relying on work of Whitney in~\cite{whitney}, Kauffman reformulated the Four Color Theorem using the vector cross product in~\cite{kauffman}. More recently, in~\cite{crz}, Cooper, Rowland and Zeilberger transformed the Four Color Theorem into a question about another binary function on $\mathbb{T}_n$: the size of the set $\text{ParseWords}(T_1, T_2)$ consisting of all words $w \in \{0, 1, 2\}^n$ which are \textbf{parsed} by both $T_1$ and $T_2$. Here, a word $w$ is parsed by $T$ if the labeling of the leaves of $T$ by $w_1, w_2, \ldots, w_n$ from left to right extends to a proper $3$-coloring with colors $\{0, 1, 2\}$ of \emph{all} $2n - 1$ vertices in $T$, such that no two children of the same vertex have the same label and such that no parent and child share the same label. The Four Color Theorem is equivalent to the statement that for all $n$ and all $T_1, T_2 \in \mathbb{T}_n$, one has $|\text{ParseWords}(T_1, T_2)| \geq 1$. Tamari offers a similar reformulation of the Four Color Theorem in~\cite{tamari}.

This last application to the Four Color Theorem motivated us to investigate a poset $\mathscr{C}_n$ on the set $\mathbb{T}_n$, which we call the \textbf{(right) comb order}, a weakening of the Tamari order. Pallo first defined $\mathscr{C}_n$ in \cite{palloposet}, where he proved that it is a meet-semilattice having the same bottom element as $\mathscr{T}_n$, called the right comb tree and denoted $\RCT{n}$. %In fact (see Corollary \ref{cor: poset = minimum} below), one way to think about the poset $\mathscr{C}_n$ is that $T_1 <_{\mathscr{C}_n} T_2$ exactly when $T_1$ lies on a geodesic path in the rotation graph $\mathscr{R}_n$ from $T_2$ to RightCombTree$(n)$. In the diagram below, the edges in geodesic paths in $\mathscr{T}_n$ are bolded and form the Hasse diagram of $\mathscr{C}_4$. The dashed edge lies in $\mathscr{T}_4$ but not $\mathscr{C}_4$.
The solid edges in the diagram below form the Hasse diagram of $\mathscr{C}_4$. The dashed edge lies in $\mathscr{T}_4$ but not in $\mathscr{C}_4$.

\begin{center}
\begin{tikzpicture}
 
    \tikzstyle{every node}=[draw,circle,fill=black,minimum size=4pt,
                            inner sep=0pt]
    
    \draw(0,1) node(A) {};
    \draw(-3,2.5) node(B) {};
    \draw(-3,6) node(C) {};
    \draw(0,7.5) node(D) {};
    \draw(2.75,4.25) node(E) {};
    
    \draw[very thick,blue] (A) -- (B);
    \draw[very thick,blue] (B) -- (C);
    \draw[dashed,blue] (C) -- (D);
    \draw[very thick,blue] (D) -- (E);
    \draw[very thick,blue] (E) -- (A);

    \tikzstyle{every node}=[draw,circle,fill=black,minimum size=1pt,
                            inner sep=0pt]
    
    \draw(0.25,0) node(A1) {};
    \draw(0.75,0) node(A2) {};
    \draw(0,0.25) node(A3) {};
    \draw(0.5,0.25) node(A4) {};
    \draw(-0.25,0.5) node(A5) {};
    \draw(0.25,0.5) node(A6) {};
    \draw(0,0.75) node(A7) {};
    \draw (A1) -- (A4);
    \draw (A2) -- (A4);
    \draw (A3) -- (A6);
    \draw (A4) -- (A6);
    \draw (A5) -- (A7);
    \draw (A6) -- (A7);
    
    \draw(-4,2) node(B1) {};
    \draw(-3.5,2) node(B2) {};
    \draw(-3.75,2.25) node(B3) {};
    \draw(-3.25,2.25) node(B4) {};
    \draw(-4,2.5) node(B5) {};
    \draw(-3.5,2.5) node(B6) {};
    \draw(-3.75,2.75) node(B7) {};
    \draw (B1) -- (B3);
    \draw (B2) -- (B3);
    \draw (B3) -- (B6);
    \draw (B4) -- (B6);
    \draw (B5) -- (B7);
    \draw (B6) -- (B7);
    
    \draw(-4,5.5) node(C1) {};
    \draw(-3.5,5.5) node(C2) {};
    \draw(-4.25,5.75) node(C3) {};
    \draw(-3.75,5.75) node(C4) {};
    \draw(-4,6) node(C5) {};
    \draw(-3.5,6) node(C6) {};
    \draw(-3.75,6.25) node(C7) {};
    \draw (C1) -- (C4);
    \draw (C2) -- (C4);
    \draw (C3) -- (C5);
    \draw (C4) -- (C5);
    \draw (C5) -- (C7);
    \draw (C6) -- (C7);
    
    \draw(-0.5,7.75) node(D1) {};
    \draw(0,7.75) node(D2) {};
    \draw(-0.25,8) node(D3) {};
    \draw(0.25,8) node(D4) {};
    \draw(0,8.25) node(D5) {};
    \draw(0.5,8.25) node(D6) {};
    \draw(0.25,8.5) node(D7) {};
    \draw (D1) -- (D3);
    \draw (D2) -- (D3);
    \draw (D3) -- (D5);
    \draw (D4) -- (D5);
    \draw (D5) -- (D7);
    \draw (D6) -- (D7);
    
    \draw(3,4) node(E1) {};
    \draw(3.5,4) node(E2) {};
    \draw(3.75,4) node(E3) {};
    \draw(4.25,4) node(E4) {};
    \draw(3.25,4.25) node(E5) {};
    \draw(4,4.25) node(E6) {};
    \draw(3.625,4.5) node(E7) {};
    \draw (E1) -- (E5);
    \draw (E2) -- (E5);
    \draw (E3) -- (E6);
    \draw (E4) -- (E6);
    \draw (E5) -- (E7);
    \draw (E6) -- (E7);
    
\end{tikzpicture}
\end{center}

While the comb order $\mathscr{C}_n$ is a meet-semilattice whose meet $\wedge_{\mathscr{C}_n}$ does not in general coincide with the Tamari meet $\wedge_{\mathscr{T}_n}$, it fixes several deficiencies of $\mathscr{T}_n$ noted above:
\begin{itemize}

  \item $\mathscr{C}_n$ is ranked, with exactly $\binom{n + r - 2}{r} - \binom{n + r - 2}{r - 1}$ elements of rank $r$, $0 \leq r \leq n-2$ (see Theorem \ref{thm: |rank|}).

  \item $\mathscr{C}_n$ is \textbf{locally distributive}; each interval forms a distributive lattice (see Corollary \ref{cor: distributive lattice}).

\item If $T_1$ and $T_2$ have an upper bound in $\mathscr{C}_n$ (or equivalently, if they both lie in some interval), the meet $T_1 \wedge_{\mathscr{C}_n} T_2$ and join $T_1 \vee_{\mathscr{C}_n} T_2$ are easily described combinatorially in two different ways (see Corollary \ref{cor: distributive lattice} and Theorem \ref{thm: bracketing vectors}). These operations also coincide with the Tamari meet $\wedge_{\mathscr{T}_n}$ and Tamari join $\vee_{\mathscr{T}_n}$ (see Corollary \ref{cor: tamari meet = Tn meet}).

\item When trees $T_1, T_2$ have an upper bound in $\mathscr{C}_n$, one has (see Theorem \ref{thm: T1 + T2 - 2TR})
\begin{align*}
d_{\mathscr{R}_n}(T_1, T_2) &= \text{rank}(T_1) + \text{rank}(T_2) - 2 \cdot \text{rank}(T_1 \wedge_{\mathscr{C}_n} T_2) \\
&= 2 \cdot \text{rank}(T_1 \vee_{\mathscr{C}_n} T_2) - \left(\text{rank}(T_1) + \text{rank}(T_2)\right)\\
&=\text{rank}(T_1\vee_{\mathscr{C}_n}T_2)-\text{rank}(T_1\wedge_{\mathscr{C}_n}T_2),
\end{align*}
where, for any $T \in \mathbb{T}_n$, rank$(T)$ refers to the rank of $T$ in $\mathscr{C}_n$.

\item Furthermore, for $T_1, T_2$ having an upper bound in $\mathscr{C}_n$, one has (see Theorem \ref{thm: join-meet})
\begin{equation*}
\text{ParseWords}(T_1, T_2) = \text{ParseWords}(T_1 \wedge_{\mathscr{C}_n} T_2, T_1 \vee_{\mathscr{C}_n} T_2),
\end{equation*}
with cardinality $3\cdot 2^{n - 1 - k}$, where $k = \text{rank}(T_1 \vee_{\mathscr{C}_n} T_2) - \text{rank}(T_1 \wedge_{\mathscr{C}_n} T_2)$ (see Theorem \ref{thm: comparable PW}).

\end{itemize}
Lastly, Section~\ref{Section 6} discusses a well-known order-preserving surjection from the (right) weak order on the symmetric group $\mathfrak{S}_n$ to the Tamari poset $\mathscr{T}_{n + 1}$ and its restriction to an order-preserving surjection from $\mathscr{E}_n$ to $\mathscr{C}_{n + 1}$ (where $\mathscr{E}_n$ is a subposet of the weak order considered by Edelman in \cite{edelman}). Furthermore, this surjection is a distributive lattice morphism on each interval of $\mathscr{C}_{n+1}$ (see Theorem \ref{thm: B-Pdual}).

Because we will be mainly confining our attention for the rest of this paper to the poset $\mathscr{C}_n$, we will drop the subscripts from $\wedge$, $\vee$, $>$ and $<$ when we mean meet, join, greater than, and less than in $\mathscr{C}_n$ respectively. Furthermore, we will use rank$(T)$ to denote the rank of $T$ in $\mathscr{C}_n$. Much of our notation in Section \ref{Section 7} is from \cite{crz}.

\section{The Comb Poset and Distributivity} \label{Section 2} \label{sec:combposetdistrib}
In~\cite{ht:bracketingvectors}, Huang and Tamari describe the dual to the Tamari lattice in terms of \emph{binary bracketings}, which are the usual parenthesizations of the leaves of a binary tree. However, the comb poset is most readily defined in terms of a variation on this parenthesization.

First, recall the definition of the parenthesization of a binary tree.

\begin{defn}\label{def:fullparenthesization}
Suppose $T \in \bT_n$ and its leaves are labeled $a_1,\ldots,a_n$. The \emph{parenthesization} of $T$ is a set $\mathcal{P}$, whose elements are the subsets, $J$, of $\{a_1,\ldots,a_n\}$ such that $J=\{a_i<\ldots<a_j\}$ and $a_i<\ldots<a_j$ label the leaves of a subtree of $T$.
\end{defn}

\begin{prop}
	Suppose $T \in \bT_n$. Then, for $E \in \mathcal{P}$, either $|E|=1$ or $E=E_1\sqcup E_2$, where $E_1,E_2 \in \mathcal{P}$, $E_1 \cap E_2 =\emptyset$ and $E_1,E_2$ are unique. 
\end{prop}

\begin{defn}\label{defn: reduced parenthesization}
	For each $T \in \mathbb{T}_n$, the \emph{reduced parenthesization} of $T$ is denoted $RP_T$ and $RP_T=\mathcal{P}\smallsetminus \{ E \in \mathcal{P}| a_n \in E\}$. An element $E$ of the set $RP_T$ with $|E|>1$ is called a \emph{parenthesis pair}.
\end{defn}

As the name suggests, one can write the parenthesization and reduced parenthesization as parenthesizations of the sequence $a_1a_2\cdots a_n$, with the singleton sets of $\mathcal{P}$ and $RP_T$ not drawn. This convention will be used for the remainder of the paper. 

\begin{figure}[htbp]
\begin{center}
\begin{tikzpicture}

    \tikzstyle{every node}=[draw,circle,fill=black,minimum size=4pt,
                            inner sep=0pt]
    \draw (-1,0.5) node (21) {};
    \draw (0,0.5) node (22) [label=above:$a_4$]{};
    \draw (-0.5,1) node (23) {};
    \draw (0.5,0.5) node (25) {};
    \draw (1,1) node (26) {};
    \draw (0,0) node (27) [label=below:$a_5$]{};
    \draw (1.5,0.5) node (28) [label=below:$a_7$]{};
    \draw (1,0) node (29) [label=below:$a_6$]{};
    \draw (-0.5,0) node (30) [label=below:$a_3$]{};
    \draw (-1.5,0) node (31) [label=below:$a_2$]{};
    \draw (-0.25,2) node (32) {};
    \draw (-0.75,1.5) node (33) [label=below:$a_1$]{};
    \draw (0.25,1.5) node (24) {};
    \draw[blue] (21) -- (23);
    \draw[blue] (22) -- (23);
    \draw[blue] (25) -- (26);
    \draw[blue] (25) -- (27);
    \draw[blue] (23) -- (24);
    \draw[blue] (24) -- (26);
    \draw[blue] (26) -- (28);
    \draw[blue] (25) -- (29);
    \draw[blue] (21) -- (31);
    \draw[blue] (21) -- (30);
    \draw[blue] (24) -- (32);
    \draw[blue] (32) -- (33);

\end{tikzpicture}
\caption{}
\label{fig: example parenthesization}\label{fig:parenextree}
\end{center}
\end{figure}

\begin{example}\label{ex: reduced par}
	The full parenthesization of the tree in Figure~\ref{fig:parenextree} is $(a_1(((a_2a_3)a_4)((a_5a_6)a_7)))$ and its reduced parenthesization is $a_1((a_2a_3)a_4)(a_5a_6)a_7$.
\end{example}

\begin{remark}All $n$-leaf binary trees have a unique reduced parenthesization, since there is a bijection between the full parenthesization of a tree $T$ and its reduced parenthesization. The full parenthesization is recovered by pairing the two rightmost elements of $RP_T$ successively.
\end{remark}

\begin{prop}\label{prop: characterize RPTs}
	A collection $\mathcal{P}$ of subsets of $\{a_1,\ldots,a_n\}$ is the reduced parenthesization of a tree $T \in \mathbb{T}_n$ if and only if the following conditions hold:
	\begin{enumrom}
		\item for each $a_i$, $1\leq i < n$, there is a $P \in \mathcal{P}$ such that $a_i \in P$ and there is no $P \in \mathcal{P}$ such that $a_n \in P$
		\item for each $P \in \mathcal{P}$, if $a_i,a_k \in P$ and $i < j <k$, then $a_j \in P$
		\item for each $P \in \mathcal{P}$, either $|P|=1$ or there are $P_1,P_2 \in \mathcal{P}$, with $P_1 \cap P_2 =\emptyset$, such that $P =P_1 \sqcup P_2$
		\item if $P_1,P_2 \in \mathcal{P}$ with $P_1 \cap P_2 \ne \emptyset$, then either $P_1 \subset P_2$ or $P_2 \subset P_1$.
	\end{enumrom}
\end{prop}

When $P \in RP_T$ can be written as $P_1 \sqcup P_2$ for $P_1,P_2 \in RP_T$, $P_1$ and $P_2$ are called the \emph{factors} of $P$. Without loss of generality assume that for each $a_i \in P_1$ and $a_j \in P_2$ that $i<j$. Then $P_1$ is the \emph{left factor} and $P_2$ is the \emph{right factor} of $P$.

\begin{defn} \label{defn: RCT}
  For $n \geq 2$, the \textbf{right comb tree of order $n$}, denoted by $\RCT{n} \in \mathbb{T}_n$, is the $n$-leaf binary tree with $RP_T=\emptyset$, corresponding to $a_1a_2 \cdots a_n$. Similarly, the \textbf{left comb tree} of order $n$ is defined as the $n$-leaf binary tree corresponding to the reduced parenthesization $RP_T=\{\{a_1,\ldots,a_i\}\colon i=1,\ldots,n-1\}$, corresponding to $(((\cdots((a_1a_2)a_3)\cdots )a_{n - 2})a_{n - 1})a_n$.
\end{defn}

\begin{example}
  \RCT{5}, the right comb tree of order $5$, is shown below. The nodes labeled $a_1, \ldots, a_5$ are the leaves of the tree, and $b_6, \ldots, b_9$ are the internal vertices. Note that the structure of the left comb tree of order $5$ is given by the reflection of the right comb tree about the vertical axis.
\begin{center}
\begin{tikzpicture}
    \tikzstyle{every node}=[draw,circle,fill=black,minimum size=4pt,
                            inner sep=0pt]

    \draw (-1,1.5) node (01) [label=below:$a_1$]{};
    \draw (-0.5,2) node (02) [label=right:$b_6$]{};
    \draw (-0.5,1) node (03) [label=below:$a_2$]{};
    \draw (0,1.5) node (04) [label=right:$b_7$]{};
    \draw (0,0.5) node (05) [label=below:$a_3$]{};
    \draw (0.5,1) node (06) [label=right:$b_8$]{};
    \draw (0.5,0) node (07) [label=below:$a_4$]{};
    \draw (1,0.5) node (08) [label=right:$b_9$]{};
    \draw (1.5,0) node (09) [label=below:$a_5$]{};
    \draw[blue] (01) -- (02);
    \draw[blue] (03) -- (04);
    \draw[blue] (05) -- (06);
    \draw[blue] (07) -- (08);
    \draw[blue] (02) -- (04);
    \draw[blue] (04) -- (06);
    \draw[blue] (06) -- (08);
    \draw[blue] (08) -- (09);
    
\end{tikzpicture}
%\caption{RightCombTree$(5)$}
\label{RCT(4)}
\end{center}
\end{example}

\begin{defn}\label{def:combposet}
	For $n \geq 2$, the \emph{(right) comb poset} of order $n$ is the poset whose elements are $\mathbb{T}_n$ and with $T_1 \leq T_2$ if $RP_{T_1} \subseteq RP_{T_2}$.
\end{defn}

\begin{remark}
	One sees immediately that $\RCT{n}$ is the unique minimal element of $\mathscr{C}_n$ since its reduced parenthesization is the empty set.
\end{remark}

\begin{example}
The Hasse diagram of the right comb poset of order $5$ is shown in Figure \ref{fig: T_5}. For the sake of a cleaner diagram, the leaf $a_i$ is labeled by $i$ in Figure \ref{fig: T_5} for $i \in \{1, 2, 3, 4, 5\}$.
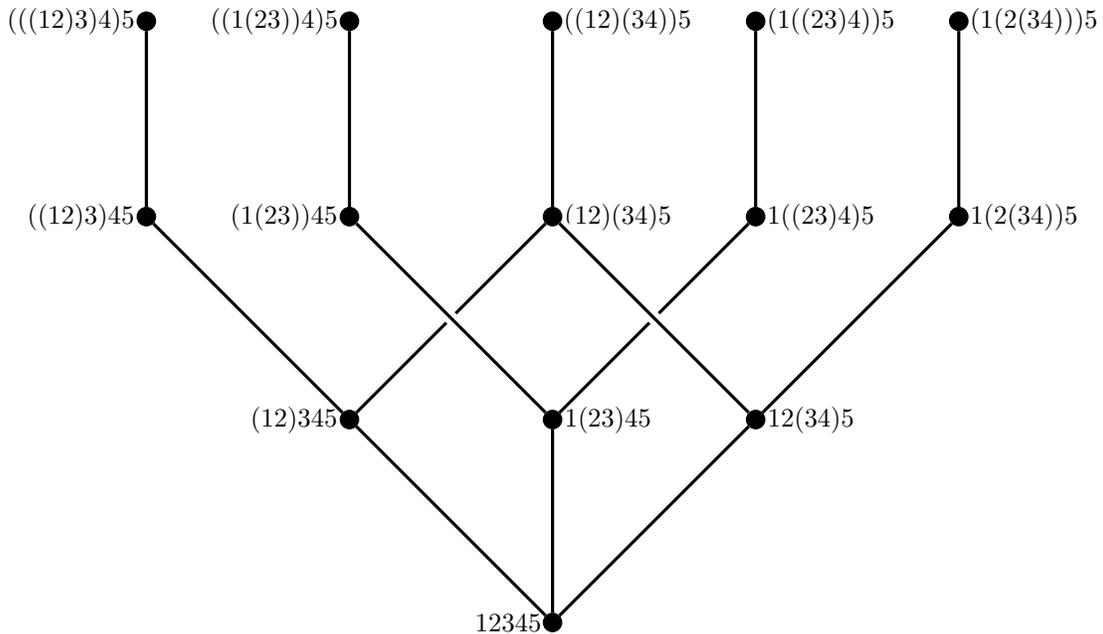
\begin{figure}[ht]
\begin{center}
\begin{tikzpicture}
    \tikzstyle{every node}=[draw,circle,fill=black,minimum size=7pt,
                            inner sep=0pt]
    \tikzstyle myBG=[line width=5pt,opacity=1.0]

\newcommand{\drawLine}[2]
{
    \draw[white,myBG]  (#1) -- (#2);
    \draw[black,very thick] (#1) -- (#2);
}

    \draw (0,0) node (00) [label=left:$12345$]{};
    \draw (-2.7,2.7) node (01) [label=left:$(12)345$]{};
    \draw (0,2.7) node (03) [label=right:$1(23)45$]{};
    \draw (2.7,2.7) node (02) [label=right:$12(34)5$]{};
    \draw (-5.4,5.4) node (04) [label=left:$((12)3)45$]{};
    \draw (-2.7,5.4) node (07) [label=left:$(1(23))45$]{};
    \draw (0,5.4) node (05) [label=right:$(12)(34)5$]{};
    \draw (2.7,5.4) node (08) [label=right:$1((23)4)5$]{};
    \draw (5.4,5.4) node (06) [label=right:$1(2(34))5$]{};
    \draw (-5.4,8) node (09) [label=left:$(((12)3)4)5$]{};
    \draw (-2.7,8) node (12) [label=left:$((1(23))4)5$]{};
    \draw (0,8) node (10) [label=right:$((12)(34))5$]{};
    \draw (2.7,8) node (13) [label=right:$(1((23)4))5$]{};
    \draw (5.4,8) node (11) [label=right:$(1(2(34)))5$]{};

    \draw[very thick] (00) -- (01);
    \draw[very thick] (00) -- (03);
    \draw[very thick] (00) -- (02);
    \draw[very thick] (01) -- (04);
    \draw[very thick] (01) -- (05);
    \drawLine{03}{07};
    \draw[very thick] (03) -- (08);
    \drawLine{02}{05};
    \draw[very thick] (02) -- (06);
    \draw[very thick] (04) -- (09);
    \draw[very thick] (07) -- (12);
    \draw[very thick] (05) -- (10);
    \draw[very thick] (08) -- (13);
    \draw[very thick] (06) -- (11);

\end{tikzpicture}
\caption{The Hasse diagram of $\mathscr{C}_5$}
\label{fig: T_5}
\end{center}
\end{figure}
\end{example}

\begin{prop} \label{prop: symmetry}
There is an order-preserving involution on $\mathscr{C}_n$.
\end{prop}

\begin{proof}
  For any tree $T$, take $RP_T$, and construct a new parenthesization $RP_{T'}$ as follows. For every parenthesis pair in $RP_T$ which encloses leaves $a_i$ through $a_j$, take $RP_{T'}$ to have a parenthesis pair enclosing leaves $a_{n - j}$ through $a_{n - i}$. It is not hard to see (using Proposition \ref{prop: characterize RPTs}) that $RP_{T'}$ corresponds to a tree $T'$. Define $\pi$ to be the map that takes $T$ to $T'$ as described above. Then, $\pi$ is an order preserving involution on $\mathscr{C}_n$.
\end{proof}

\begin{defn} \label{defn: right arm}
For n $\geq 2$, the \textbf{right arm} of a tree $T \in \mathbb{T}_n$ is the path induced by the vertices of $T$ that lie in the left subtree of no other vertex in $T$.
\end{defn}

To understand the properties of the intervals of $\mathscr{C}_n$, one needs to define another poset using $RP_T$. It is well known that the operation of ``pruning'' a tree, i.e.\ deleting the leaves, is a bijection between $n$-leaf binary trees and (possibly incomplete) binary trees with $n-1$ vertices.

\begin{defn}\label{def: rpposet}
	For a tree $T \in \mathbb{T}_n$, the \emph{reduced pruned poset} of $T$, denoted $P_T$ is the poset obtained by ordering by inclusion those elements of $RP_T$ which are not singleton sets. Its Hasse diagram is obtained by pruning $T$, removing the right arm and removing those edges incident to the right arm.
\end{defn}

\begin{example}
Consider the tree of Figure \ref{fig: example parenthesization}, given by reduced parenthesization $a_1((a_2a_3)a_4)(a_5a_6)a_7$. Figure~\ref{fig: treeposets} depicts its ``pruned'' form, the corresponding reduced pruned poset $P_T$ and $J(P_T)$.
\end{example}
\begin{figure}
  \begin{center}
    \begin{tikzpicture}
    \tikzstyle{vertex}=[draw,circle,fill=black,minimum size=4pt,
                                inner sep=0pt]
    \matrix{
      & \begin{scope}
      \node[vertex] (root) at (0,0) {};
      \node[vertex,below left of=root,label=left:$a_1$,xshift=-.25cm] (a1) {};
      \node[vertex,below right of=root,xshift=.25cm] (arm1) {};
      \node[vertex,below left of=arm1,xshift=-.25cm] (1) {};
      \node[vertex,below right of=1,label=left:$a_4$] (a4) {};
      \node[vertex,below left of=1] (2) {};
      \node[vertex,below left of=2,label=below:$a_2$,xshift=.2cm] (a2) {};
      \node[vertex,below right of=2,label=below:$a_3$,xshift=-.2cm] (a3) {};
      \draw (a1)--(root)--(arm1)--(1)--(2)--(a2);
      \draw(2)--(a3);
      \draw(1)--(a4);
      \node[vertex,below right of=arm1,xshift=.25cm] (arm2) {};
      \node[vertex,below left of=arm2] (3) {};
      \node[vertex,below right of=arm2,label=below:$a_7$] (a7) {};
      \node[vertex,below left of=3,label=below:$a_5$,xshift=.2cm] (a5) {};
      \node[vertex,below right of=3,label=below:$a_6$,xshift=-.2cm] (a6){};
      \draw(arm1)--(arm2)--(a7);
      \draw(arm2)--(3)--(a5);
      \draw(3)--(a6);
      \node[left of=a1,xshift=.25cm,yshift=-1cm] {$T$};
    \end{scope}
      & \\
      \begin{scope}[every node/.style={draw,circle,fill=black,minimum size=4pt,
                                  inner sep=0pt}]
        \node (root) at (0,0) {};
        \node[below right of=root] (1) {};
        \node[below left of=1] (2) {};
        \node[below right of=1] (3) {};
        \node[below left of=2] (4) {};
        \node[below left of=3] (5) {};
        \draw (root)--(1)--(2)--(4);
        \draw (1)--(3)--(5);
        \node[draw=none,fill=none,below right of=4] {$T$ pruned};
        %\draw (-1,0.5) node (21) {};
        %\draw (-0.5,1) node (23) {};
        %\draw (0,0.5) node (25) {};
        %\draw (0.5,1) node (26) {};
        %\draw (-0.5,2) node (32) {};
        %\draw (0,1.5) node (24) {};
        %\draw (21) -- (23);
        %\draw (25) -- (26);
        %\draw (23) -- (24);
        %\draw (24) -- (26);
        %\draw (24) -- (32);
      \end{scope} & &
      \begin{scope}[every node/.style={draw,circle,fill=black,minimum size=4pt,
                                  inner sep=0pt}]
        \draw (0,0) node (X2) [label=right:$(a_2a_3)a_4$] {};
        \node[below left of=X2,label=right:$a_2a_3$] (X1) {};
        \node[below right of=X2,label=right:$a_5a_6$] (X3) {};
        \draw (X1)--(X2);
        \node[draw=none,fill=none,below right of=X1] {$P_T$};
        %\draw (0,1) node (X1) [label=right:$a_2a_3$]{};
        %\draw (1,2) node (X3)[label=right:$(a_2a_3)a_4$] {};
        %\draw (2,1) node (X5) [label=right:$a_5a_6$] {};
        %\draw (X1) -- (X3);
      %\node[draw=none,fill=none] (0,1) [label=below:$P_T$]{};
      \end{scope} \\
      & \begin{scope}
      \tikzstyle{every node}=[draw,circle,fill=black,minimum size=4pt,
                                  inner sep=0pt]
      	\draw (1,0) node (01) [label=below:$\emptyset$]{};
      	\draw (0,1) node (02) [label=left:$\{a_2a_3\}$]{};
      	\draw (2,1) node (03) [label=right:$\{a_5a_6\}$]{};
      	\draw (-1,2) node (04) [label=left:$\{a_2a_3\text{,}(a_2a_3)a_4\}$]{};
      	\draw (1,2) node (05) [label=right:$\{a_2a_3\text{,}a_5a_6\}$]{};
      	\draw (0,3) node (06) [label=right:$\{a_2a_3\text{,}(a_2a_3)a_4\text{,}a_5a_6\}$]{};
      	\draw (01)--(02);
      	\draw (01)--(03);
      	\draw (02)--(04);
      	\draw (02)--(05);
      	\draw (03)--(05);
      	\draw (04)--(06);
      	\draw (05)--(06);
      	\node[draw=none,fill=none] (0,0) [label=below:$J(P_T)$]{};
    \end{scope}
    & \\
      };
    \end{tikzpicture}
      %%%%%%%%%%%%%%%%%%%%%%%%%%%%%%
\end{center}
\caption{}
\label{fig: treeposets}
\end{figure}

\begin{prop}\label{prop: elements = left subtrees}
For any tree $T \in \mathbb{T}_n$, the maximal elements of $P_T$ correspond to the left subtrees of the vertices of the right arm of $T$. 
	%For any tree $T \in \mathbb{T}_n$, the left subtrees of vertices on the right arm of $T$ correspond to the maximal elements of $P_T$ and those $a_i$ that do not appear in any element of $RP_T$.
\end{prop}

\begin{prop}\label{prop: JPT interval}
  For any $T \in \mathbb{T}_n$, the interval $[\RCT{n},T]_{\mathscr{C}_n}$ is isomorphic to the lattice of order ideals in the reduced pruned poset of $T$, ordered by inclusion. In other words, for any tree $T$,
\[
[\RCT{n},T]_{\mathscr{C}_n}\cong J(P_T)
\]
\end{prop}

\begin{proof}One has a natural map $J(P_T)\to [\RCT{n},T]$, given by $I \mapsto S$, where $S$ is the tree with $RP_S$ having precisely the parentheses in $I$. The definition of the order on $\C_n$ ensures this map is both well-defined and order-preserving. Furthermore, this map has an inverse $[\RCT{n},T]\to J(P_T)$ given by $S\mapsto \{E \in RP_S \colon |E|>1\}$, which is again order-preserving.
\end{proof}

This proposition yields a number of immediate corollaries.
\begin{cor}\ 
  \begin{enumrom}[ref={\thecor(\roman*)}]
	\item \label{cor: distributive lattice}
 Any interval in $\mathscr{C}_n$ is a distributive lattice, with the reduced parenthesizations of the join and meet of trees $T_1$ and $T_2$ in an interval given by the ordinary union and intersection of parenthesis pairs from $RP_{T_1}$ and $RP_{T_2}$.

	\item \label{cor: cover in Tn}
In $\mathscr{C}_n$, $T_1$ covers $T_2$ if and only if $RP_{T_1}$ can be obtained from $RP_{T_2}$ by adding one parenthesis pair.

	\item  \label{cor: rank = PP}
$\mathscr{C}_n$ is a ranked poset, with the rank of any tree $T$ in $\mathscr{C}_n$ given by the number of parenthesis pairs in $RP_T$.

	\item \label{cor: r(T1) + r(T2)}
 For any two trees $T_1$ and $T_2$ that are in the same interval of $\mathscr{C}_n$, we have 
 \[
 \text{rank}(T_1) + \text{rank}(T_2) = \text{rank}(T_1 \wedge T_2) + \text{rank}(T_1 \vee T_2)
 \]

	\item \label{cor: c + k = n - 1}
 For any tree $T \in \mathbb{T}_n$ of rank $k$, the length of the right arm of $T$ is $ n - 1 - k$.
  \end{enumrom}
\end{cor}

\begin{remark}
  It is important to note that Corollary~\label{cor:distributive lattice} presumes pairs of parentheses have knowledge of their factors when talking about the ``ordinary'' union and intersection. For example, the trees with reduced parenthesizations $(a(bc))d$ and $((ab)c)d$ do not have a join, as, in one case, the factors of $\{a,b,c\}$ are $\{a\}$ and $\{b,c\}$, while in the other the factors are $\{a,b\}$ and $\{c\}$. Similarly, the meet of these two trees is the right comb tree, having reduced parenthesization $abcd$.
\end{remark}

\begin{note}
In the remainder of this paper, we shall consider only the right comb poset of order $n$. Analogous results hold for the left comb poset by symmetry.
\end{note}

\begin{remark}\label{rmk: pallodefs}

  A \emph{left rotation} is the following operation on a tree, which takes place in a subtree with root $r$:
  \begin{center}
   \begin{tikzpicture}
    \node[draw,circle,fill=black,minimum size=4pt,inner sep=0pt] (r) at (0,2) [label=above:$r$] {};
    \node[draw,circle,fill=black,minimum size=4pt,inner sep=0pt] (i) at (1,1) {};
    \node[draw,circle,fill=white, minimum size=10pt, inner sep=0pt] (A) at (-1,1) [label=below:$A$] {};
    \node[draw,circle,fill=white, minimum size=10pt, inner sep=0pt] (B) at (0,0) [label=below:$B$] {};
    \node[draw,circle,fill=white, minimum size=10pt, inner sep=0pt] (C) at (2,0) [label=below:$C$] {};
    \draw (r)--(i);
    \draw (r)--(A);
    \draw (i)--(B);
    \draw(i)--(C);

    \path[draw,thick,black,->] (2,1) -- (4,1);

      \node[draw,circle,fill=black,minimum size=4pt,inner sep=0pt] (r1) at (6,2) [label=above:$r$] {};
    \node[draw,circle,fill=black,minimum size=4pt,inner sep=0pt] (i1) at (5,1) {};
    \node[draw,circle,fill=white, minimum size=10pt, inner sep=0pt] (A1) at (4,0) [label=below:$A$] {};
    \node[draw,circle,fill=white, minimum size=10pt, inner sep=0pt] (B1) at (6,0) [label=below:$B$] {};
    \node[draw,circle,fill=white, minimum size=10pt, inner sep=0pt] (C1) at (7,1) [label=below:$C$] {};
    \draw (r1)--(i1);
    \draw (r1)--(C1);
    \draw (i1)--(B1);
    \draw(i1)--(A1);
   \end{tikzpicture}
  \end{center}
  \emph{Right arm rotations} are those where $r$ lies on the right arm of the tree.	The covering relation described in Corollary~\ref{cor: cover in Tn} corresponds right arm rotation, precisely the covering relation used by Pallo in~\cite{palloposet} to define the poset $(B_n, \stackrel{*}{\leadsto})$, which he showed to be a meet-semilattice~~\cite[Lemma 3]{palloposet}.
\end{remark}

\section{Rank Sizes in the Comb Poset} \label{Section 4}\label{sec: properties}\label{sec: rank sizes}
%Note: This is really the third section
In this section, we will prove some enumerative properties of the ranks of $\mathscr{C}_n$. To simplify notation, let $Q_i$ denote the $i$\textsuperscript{th} rank of $\mathscr{C}_n$.

\begin{prop} \label{prop: up maps}
For $0 \leq i \leq n - 2$, every tree in $Q_i$ is covered by precisely $n - 2 - i$ trees.
\end{prop}

\begin{proof}
This fact follows from the definition of rotation, and the observation that a tree in rank $Q_i$ has, by Corollary \ref{cor: c + k = n - 1}, a right arm of length $n - 1 - i$, i.e.\ $n - i$ vertices.
\end{proof}

\begin{thm} \label{thm: |rank|}
  \label{prop: maxrank}
  For $n \geq 3$, $\mathscr{C}_n$ is a ranked poset. A tree is a maximal element of $\mathscr{C}_n$ if and only if it is of rank $n-2$ in $\mathscr{C}_n$ (or equivalently, from Corollary \ref{cor: c + k = n - 1}, if and only if its right arm has length $1$). In particular, the left comb tree is in the maximal rank of $\mathscr{C}_n$. Furthermore, for $0 \leq r \leq n - 2$, the number of elements in rank $r$ of $\mathscr{C}_n$ is $$|Q_r| = \binom{n + r - 2}{r} - \binom{n + r - 2}{r - 1}.$$
\end{thm}

The authors thank an anonymous referee for pointing out the following combinatorial proof.

\begin{proof}
	Suppose $n \geq 3$ and $T$ is in rank $r$ of $\mathscr{C}_n$. Then $RP_T$ has $r$ pairs of parentheses, which are completely determined by the open parentheses as a consequence of Proposition~\ref{prop: characterize RPTs}. Furthermore, when viewing $RP_T$ as a parenthesization of $a_1\cdots a_n$ and reading from right to left there are always at least two more $a_i$ than open parentheses, as there can be no open parenthesis immediately preceding either $a_{n-1}$ or $a_n$. One may read off a lattice path from $(0,0)$ to $(n-1,r)$ that touches the line $y=x$ only at $(0,0)$ as follows. One deletes the closed parentheses and $a_n$ from $RP_T$, leaving a string consisting of $a_i$, $i \in \{1,\ldots,n-1\}$, and open parentheses. One obtains a lattice path by reading from right to left and recording an east step for each $a_i$ and a north step for each $($. (Having deleted $a_n$ means that there will only be $n-1$ east steps and that there will always have been at least one more east step than north step.)
The number of such paths is well-known (see, for example,~\cite[Exercise 6.20b]{stanleyec2}) and one has
	\begin{align*}
	  |Q_r| &=\dfrac{n-1-r}{n-1+r}\binom{n-1+r}{r} \\
		&=((n-1)-r)\dfrac{(n-2+r)!}{r!(n-1)!} \\
		&=\dfrac{(n-2+r)!}{r!(n-2)!}-\dfrac{(n-2+r)!}{(r-1)!(n-1)!} \\
		&=\binom{n+r-2}{r}-\binom{n+r-2}{r-1}.
	\end{align*}

\end{proof}

\begin{cor}
The sizes of the ranks in $\mathscr{C}_n$ weakly increase. In fact, they strictly increase until the final rank $Q_{n - 2}$, which has the same size, $C_{n-2}$, as the penultimate rank $Q_{n - 3}$.
\end{cor}

\begin{proof}
From Theorem \ref{thm: |rank|}, it can be seen that $|Q_i| = \frac{(n + i - 2)!}{i! (n - 1)!} \cdot (n -  i - 1)$, and so, for consecutive ranks $r$ and $r + 1$, one has 
\begin{equation*}
\frac{|Q_{r + 1}|}{|Q_r|} = \frac{(n + r - 1)(n - 2 - r)}{(r + 1)(n - 1 - r)}.
\end{equation*}
The rank size increases weakly whenever the numerator is at least as large as the denominator, and hence the condition for weakly increasing rank size is $(n + r - 1)(n - r - 2) \geq (r + 1)(n - r - 1)$. But this condition reduces after a few simple manipulations to the condition $n^2 - 4n + 3 - r(n - 1) \geq 0$. The result can be verified easily.
\end{proof}

\section{Distances in $\mathscr{C}_n$ and $\mathscr{R}_n$}\label{sec: distances}

We now prove some properties of the comb poset relating to the distance between pairs of trees in the rotation graph $\mathscr{R}_n$.

\begin{prop} \label{prop: tamari order}
  Any ascending chain in the right comb poset $\mathscr{C}_n$ is an ascending chain in $\mathscr{T}_n$.
\end{prop}

\begin{proof}
From Corollary \ref{cor: cover in Tn}, one has that $T_2$ is a cover of $T_1$ in $\mathscr{C}_n$ if and only if $RP_{T_2}$ can be obtained from $RP_{T_1}$ by adding precisely one more parenthesis pair. Adding any parenthesis pair to $RP_T$ is the same as shifting a pair of parentheses to the \emph{left} in the corresponding full parenthesization of the leaves of $T$.
\end{proof}

\begin{prop} \label{prop: J = J'}
Suppose $T_1$ and $T_2$ are two trees having a common upper bound in $\mathscr{C}_n$. Furthermore, suppose there are pairs of parentheses $J_1$ in $RP_{T_1}$ and $J_2$ in $RP_{T_2}$ such that $J_1$ and $J_2$ enclose a common factor. Then, $J_1 = J_2$.
\end{prop}

\begin{proof}
Suppose $T_1$ and $T_2$ have a common upper bound in $\mathscr{C}_n$ and suppose there are pairs of parentheses $J_1$ in $RP_{T_1}$ and $J_2$ in $RP_{T_2}$ enclosing a common factor, $P$. As $T_1$ and $T_2$ have a common upper bound, $T_1 \vee T_2$, from Corollary~\ref{cor: distributive lattice}, one has that $P, J_1, J_2 \in RP_{T_1 \vee T_2}$.  From Proposition~\ref{prop: characterize RPTs}, one has without loss of generality that $J_2 \subset J_1$, with $J_1=P\sqcup P_1$ and $J_2=P\sqcup P_2$. One then has that $P_1=P_2 \sqcup P_3$, for some $P_3$. Both $J_2$ and $P_1$ are in $RP_{T_1 \vee T_2}$ and have nontrivial intersection, yet neither contains the other, contradicting Proposition~\ref{prop: characterize RPTs}.
%Suppose first that the common factor $E$ is the left factor in $J_1$ and the right factor in $J_2$. Then, because $T_1 \vee T_2$ exists, from Corollary \ref{cor: distributive lattice}, one has that $J_1$ and $J_2$ must be in $RP_{T_1 \vee T_2}$ as well, which then implies that $E$ is enclosed as a single factor by a pair of parentheses in $RP_{T_1 \vee T_2}$, contradicting Proposition \ref{prop: characterize RPTs}. So, assume without loss of generality that $E$ is the \emph{left} factor for both $J_1$ and $J_2$. If we denote the right factors of $J_1$ and $J_2$ by $E_1$ and $E_2$, respectively, then these must both be the right factor of the corresponding parenthesis pair in $RP_{T_1 \vee T_2}$, forcing $E_1 = E_2$, and hence $J_1 = J_2$.
\end{proof}

\begin{lemma}\label{lem: minimal paths size of symmetric difference}
  Suppose $T_1$ and $T_2$ are two trees with a common upper bound in $\mathscr{C}_n$ and suppose $(\lambda_1, \lambda_2, \ldots, \lambda_\ell)$ is a path from $T_1$ to $T_2$ in $\mathscr{R}_n$, with $\lambda_i$ being the rotation between trees $S_i$ and $S_{i+1}$, with $S_1=T_1$ and $\lambda_\ell$ the rotation from $S_\ell$ to $T_2$. Let $RP_{T_1} \bigtriangleup RP_{T_2}$ denote the symmetric difference of $RP_{T_1}$ and $RP_{T_2}$. Let $f \colon RP_{T_1} \bigtriangleup RP_{T_2} \to \{\lambda_1,\ldots,\lambda_s\}$ be the map defined by $f(J)=\lambda_j$, where $j$ is the minimum index such that $J \in RP_{S_i} \smallsetminus RP_{S_{i+1}}$ or $J \in RP_{S_{i+1}} \smallsetminus RP_{S_i}$. Then $f$ is injective and the shortest possible length of a path from $T_1$ to $T_2$ along the edges of the rotation graph $\mathscr{R}_n$ is $|RP_{T_1}\bigtriangleup RP_{T_2}|$.
\end{lemma}

\begin{proof}
  From Corollary \ref{cor: distributive lattice} one has that $RP_{T_1 \wedge T_2}$ contains all the common parenthesis pairs of $RP_{T_1}$ and $RP_{T_2}$. Hence, $RP_{T_1}$ and $RP_{T_2}$ are formed by adding, respectively, some $r$ and $s$ \emph{extra} pairs of parentheses to $RP_{T_1 \wedge T_2}$, from Corollary \ref{cor: cover in Tn}, where $r$ and $s$ are nonnegative integers, and $|RP_{T_1}\bigtriangleup RP_{T_2}|=r+s$.

  Suppose $f(J)=f(K) =\lambda_j$. If $J \ne K$, then $\lambda_j$ is a rotation sending $J$ to $K$ or vice versa. Without loss of generality, assume $\lambda_j$ is a rotation sending $J$ to $K$. Since $J,K \in RP_{T_1} \bigtriangleup RP_{T_2}$, both are in $RP_{T_1 \vee T_2}$. However, since $\lambda_j$ is a rotation sending $J$ to $K$, one must have that $J$ and $K$ share a factor. But then Proposition~\ref{prop: J = J'} forces $J = K$, a contradiction. Thus $f$ must be injective, so the minimum length of a path from $T_1$ to $T_2$ in $\mathscr{R}_n$ is $|RP_{T_1} \bigtriangleup RP_{T_2}|$. 

  %Suppose $f(J)=f(K)=\lambda_j$. If $J \ne K$, then $\lambda_j$ is a rotation sending $J$ to $K$ or vice versa. Without loss of generality, assume $\lambda_j$ is a rotation sending $J$ to $K$, meaning $J$ and $K$ must share a common factor. One has that $\lambda_j$ is a rotation centered at the root of a subtree of $S_j$, call it $S$. Then $\lambda_j$ sends $S$ to a subtree of $S_{j+1}$, call it $S'$. For some $m \leq n$ one has that either $S \lessdot S'$ or $S' \lessdot S$ in $\mathscr{C}_m$, depending on the direction of the rotation $\lambda_j$. Without loss of generality, assume $S \lessdot S'$. Then both $J$ and $K$ are in $RP_{S'}$. Since $J$ and $K$ share a common factor, by Proposition~\ref{prop: J = J'}, one must have $J=K$, a contradiction. Thus the map $f$ is injective and the minimum length of a path is $|RP_{T_1} \bigtriangleup RP_{T_2}|$.
\end{proof}

\begin{thm} \label{thm: T1 + T2 - 2TR}
If $T_1$ and $T_2$ are two trees in some interval in $\mathscr{C}_n$, then the shortest distance between them \emph{along the edges of the rotation graph $\mathscr{R}_n$} is given by \[
d_{\mathscr{R}_n}(T_1, T_2) = \text{rank}(T_1) + \text{rank}(T_2) - 2 \cdot \text{rank}(T_1 \wedge T_2).
\]
Equivalently, from Corollary \ref{cor: r(T1) + r(T2)}, this shortest distance is also given by $$d_{\mathscr{R}_n}(T_1, T_2) = 2 \cdot \text{rank}(T_1 \vee T_2) - \text{rank}(T_1) - \text{rank}(T_2).$$
\end{thm}

\begin{proof}
  If $T_1$ and $T_2$ are two trees in some interval in $\mathscr{C}_n$, $|RP_{T_1}\bigtriangleup RP_{T_2}|=\text{rank}(T_1) + \text{rank}(T_2) - 2 \cdot \text{rank}(T_1 \wedge T_2)$. From Lemma~\ref{lem: minimal paths size of symmetric difference}, one knows that the minimal possible length of a path from $T_1$ to $T_2$ is $|RP_{T_1}\bigtriangleup RP_{T_2}|$. Furthermore, one knows a path of this length exists -- the path in $\mathscr{C}_n$ from $T_1$ to $T_1 \wedge T_2$ obtained by deleting the pairs of parentheses in $RP_{T_1}$ that do not appear in $RP_{T_2}$, followed by the path from $T_1 \wedge T_2$ to $T_2$ obtained by adding the pairs of parentheses in $RP_{T_2}$ not appearing in $RP_{T_1}$.
\end{proof}

\begin{comment}
\begin{cor} \label{cor: r + s min length}
For trees $T_1$ and $T_2$ with an upper bound in $\mathscr{C}_n$, if $RP_{T_1}$ and $RP_{T_2}$ are obtained by adding some $r$ extra pairs $\{J_i\}$ and $s$ extra pairs $\{J'_i\}$ of parentheses to $RP_{T_1 \wedge T_2}$, then
\begin{enumrom}

\item $J_i$ and $J'_j$ are disjoint for all $1 \leq i \leq r$ and $1 \leq j \leq s$.
\item Any shortest path between $T_1$ and $T_2$ in $\mathscr{R}_n$ has length $r + s$.

\item In any such shortest path from $T_1$ to $T_2$, precisely $r$ rotations will affect one of the $\{J_i\}$, precisely $s$ rotations will affect one of the $\{J_i'\}$, and no rotation can turn a $J_i$ into a $J_i'$.
\end{enumrom}
\end{cor}
\end{comment}

\begin{thm}\label{thm: shortest paths}
  For $T_1,T_2$ with an upper bound in $\mathscr{C}_n$, any shortest path in $\mathscr{R}_n$ from $T_1$ to $T_2$ also lies in $\mathscr{C}_n$.
\end{thm}

\begin{proof}
  Suppose $T_1$ and $T_2$ lie in some interval of $\mathscr{C}_n$ and $(\lambda_1,\ldots,\lambda_{r+s})$ is a shortest path from $T_1$ to $T_2$ in $\mathscr{R}_n$, with $\lambda_i$ being a rotation between trees $S_i$ and $S_{i+1}$. Suppose $\lambda_i$ is a rotation not centered on the right arm of $S_i$, i.e.\ it ``shifts'' a pair of parentheses, so $|RP_{S_i}|=|RP_{S_{i+1}}|$, and assume that $\lambda_i$ is the first such rotation. In precise terms, this means there are two pairs of parentheses $J, J'$ with $\{J\} = RP_{S_i}\smallsetminus RP_{S_{i+1}}$ and $\{J'\} = RP_{S_{i+1}}\smallsetminus RP_{S_i}$. Since $(\lambda_1,\ldots,\lambda_{r+s})$ is a path of shortest possible length, the map $f$ in Lemma~\ref{lem: minimal paths size of symmetric difference} is a bijection and one must have $J \in RP_{T_1}\smallsetminus RP_{T_2}$ or $J' \in RP_{T_2}\smallsetminus RP_{T_1}$, but not both. 
  
  Suppose $J' \in RP_{T_2}\smallsetminus RP_{T_1}$. Then there must be some rotation $\lambda_j$ with $j < i$ and $\{J\}= RP_{S_{j+1}}\smallsetminus RP_{S_j}$. However, since $J \notin RP_{T_1}$, since $f$ is a bijection, one must have that $\lambda_j$ transformed some $J'' \in RP_{T_1}\cap RP_{S_j}$ into $J$, i.e.\ $\lambda_j$ cannot be centered on the right arm of $RP_{S_j}$, contradicting that $\lambda_i$ was the first such rotation.
  
  Consequently, one must have $J \in RP_{T_1}\smallsetminus RP_{T_2}$. Without loss of generality, one may assume that $\lambda_i$ shifts $J$ to the right, i.e.\ $J=A \sqcup B$ for factor $A,B \in RP_{S_i}$ and $J'= B\sqcup C$ for factors $B,C \in RP_{S_{i+1}}$. Moreover, for the rotation $\lambda_i$ to take place, one must have that $A,B,C \in RP_{S_i} \cap RP_{S_{i+1}}$ and that there is some $K = A \sqcup B \sqcup C \in RP_{S_i} \cap RP_{S_{i+1}}$. Since $J \in RP_{T_1}$, one has that $J \in RP_{T_1\vee T_2}$, so since $J \cap J' =B$, one must have that $J' \notin RP_{T_1 \vee T_2}$. Then there is a $j>i$ such that $\lambda_j$ transforms $J'$ to some $J'' \in RP_{T_1 \vee T_2}$. Moreover, since $f$ is a bijection, one must have that $J \ne J''$.
  
  Suppose $J$ encloses the leaves $\{a_{m+1},\ldots,a_q\}$, $J'$ encloses the leaves $\{a_{p+1},\ldots,a_t\}$ and $K$ encloses $\{a_{m+1},\ldots,a_t\}$ for $0 \leq m<p<q<t<n$. $J''$ is obtained from $J'$ by a rotation, so $J''$ must enclose either $a_{p+1}$ or $a_t$. However, $a_{p+1} \in J$ and, since $J, J'' \in RP_{T_1\vee T_2}$, $J \cap J'' = \emptyset$. Consequently, $a_t \in J''$. However, since $a_t$ was the last leaf enclosed by $J'$, it cannot be the last leaf enclosed by $J''$. Recall that $a_t$ is the last leaf enclosed by $K$, so $J'' \not\subset K$. Since $K \cap J'' \ne \emptyset$ and $J''\not \subset K$, one must have that $K \notin RP_{T_1\vee T_2}$. Thus, $K$ is in neither $RP_{T_1}$ nor $RP_{T_2}$, so it must result from a rotation $\lambda_k$ with $RP_{S_k} \bigtriangleup RP_{S_{k+1}}=\{K,K'\}$, i.e.\ with $\lambda_k$ not centered on the right arm of $S_k$. But since $K \in RP_{S_i}$, $k<i$, a contradiction.

  Thus, all rotations $\lambda_j$ must be centered on the right arm of $S_j$, i.e.\ the path $(\lambda_1,\ldots,\lambda_{r+s})$ lies entirely in $\mathscr{C}_n$.

\end{proof}

%This amounts to proving that for any pair of trees $T_1$ and $T_2$ with an upper bound in $\mathscr{C}_n$, any shortest path between them in $\mathscr{R}_n$ consists of simply adding and deleting parenthesis pairs from their $RP_T$'s, and not of any ``shifting'' move. The conjecture appears to be true, but the authors have not been able to find a satisfactory argument.

\begin{cor} \label{cor: rank = distance}
The rank of any tree $T \in \mathbb{T}_n$ in $\mathscr{C}_n$ is its distance from the right comb tree along the edges of the rotation graph $\mathscr{R}_n$. Furthermore, from Corollary \ref{cor: rank = PP}, the distance of $T$ from the right comb tree in $\mathscr{R}_n$ is given by the number of parenthesis pairs in $RP_T$.
\end{cor}

\begin{remark}
  It can be easily shown from the result above that the \textbf{diameter} of the rotation graph $\mathscr{R}_n$, given by the maximum distance between any pair of trees in $\mathscr{R}_n$, is at most $2n - 4$ for any $n \in \mathbb{N}$. In~\cite{stt}, Sleator, Tarjan and Thurston established the tighter bound of $2n-6$ on the diameter of the rotation graph for $n \geq 11$.
\end{remark}

\section{Tamari Meets and Joins for Two Trees in Some Interval} \label{Section 5}\label{sec:meetjoin}\label{sec: meetjoin}

From Corollary \ref{cor: distributive lattice}, we know the meaning of the meet and join of a pair of trees having a common upper bound in our poset. It is natural to ask how these meets and joins relate to meets and joins in the Tamari lattice. As before, we will refers to meets and joins in the Tamari lattice $\mathscr{T}_n$ as the ``Tamari meet'' and ``Tamari join''. 

The first observation is that, while two arbitrary trees in $\mathscr{C}_n$ \emph{do} have a well-defined meet in $\mathscr{C}_n$, this meet does not necessarily correspond to the Tamari meet. For example, consider the pair of trees represented by $T_1 = (((a_1a_2)a_3)a_4)a_5$ and $T_2 = ((a_1(a_2a_3))a_4)a_5$. This pair has Tamari meet $T_2$, while their meet in $\mathscr{C}_n$ is just the right comb tree. Further, recall that $\mathscr{C}_n$ is a meet-semilattice rather than a lattice, so not all pairs of trees have a join.

However, something much stronger can be said if both the trees under consideration are in some interval in the comb poset; it turns out that their meet and join in $\mathscr{C}_n$ correspond to their Tamari meet and join.

In~\cite{ht:bracketingvectors}, Huang and Tamari consider the lattice dual to $\mathscr{T}_n$ and characterize the meet in that lattice as the componentwise minimum of the \emph{bracketing vectors}. In~\cite{pallovectors}, Pallo obtains an analogous result for $\mathscr{T}_n$ in terms of \emph{weight vectors}, which will be of use here. 

\begin{defn}\label{def:weightvector}
	Suppose $T \in \mathbb{T}_n$. For each $i \in \{1,\ldots,n-1\}$, let $w_T(i)=\max_{E\in RP_T \colon i =\max E} |E|$. The \emph{weight vector} of $T$ is $\langle T \rangle =\langle w_T(i)\rangle$.
\end{defn}

\begin{example}
Consider the tree $T$ having reduced parenthesization $( (a_1a_2)(a_3a_4))a_5(a_6(a_7a_8))a_9$. For illustrative purposes, enclose each $a_i$ in a pair of parentheses to represent the singleton sets in $RP_T$, giving 
\[
( ( (a_1)(a_2))( (a_3)(a_4)))(a_5)( (a_6)((a_7)(a_8)))(a_9).
\]
Then $\langle T \rangle =(1,2,1,4,1,1,1,3)$.
\end{example}

\begin{thm}[{Pallo, \cite[Theorem 2]{pallovectors}}] \label{thm:pallo bracketing}
For two $n$-leaf binary trees $T$ and $T'$, one has $T\leq T'$ if and only if the weight vector of $T$ is component-wise less than or equal to the weight vector of $T'$. Furthermore, the bracketing vector for the meet of two trees in the Tamari lattice corresponds to the componentwise minimum of the weight vectors of the two trees.
\end{thm}

\begin{thm} \label{thm: bracketing vectors}
Let $\langle T \rangle$ denote the weight vector of $T \in \bT_n$. Let $T_1$ and $T_2$ be arbitrary trees in the same interval of $\mathscr{C}_n$. Then, their meet and join in $\mathscr{C}_n$ are given by the trees corresponding respectively to the componentwise minimum and the componentwise maximum of $\langle T_1 \rangle$ and $\langle T_2 \rangle$.
\end{thm}

\begin{proof}
First, consider $\langle T_1 \vee T_2\rangle$. Suppose the $i$th coordinate is $k$. Then, by definition, $k=\max_{E \in RP_{T_1\vee T_2} \colon i \in E} |E|$. From Corollary~\ref{cor: distributive lattice}, one has that $RP_{T_1 \vee T_2}=RP_{T_1}\cup RP_{T_2}$, so one must have that $k=\max(w_{T_1}(i),w_{T_2}(i))$. In other words, $\langle T_1 \join T_2 \rangle$ is the componentwise maximum of $\langle T_1 \rangle$ and $\langle T_2 \rangle$. The proof for $T_1 \meet T_2$ is analogous.
\end{proof}

\begin{cor} \label{cor: tamari meet = Tn meet}
For $T_1$ and $T_2$ in some interval in $\mathscr{C}_n$, their meet and join in $\mathscr{C}_n$ correspond respectively to their meet and join in the Tamari lattice $\mathscr{T}_n$.
\end{cor}

\begin{proof}
The proof for the meet follows directly from Theorems~\ref{thm:pallo bracketing} and~\ref{thm: bracketing vectors}. For the join, observe that the tree corresponding to the componentwise maximum of $\langle T_1 \rangle$ and $\langle T_2 \rangle$ would be the join of $T_1$ and $T_2$ in $\mathscr{T}_n$. However, in general, one does not know that such a tree exists. However, Theorem~\ref{thm: bracketing vectors} gives that such a tree exists---it is the join of $T_1$ and $T_2$ in $\mathscr{C}_n$.
\end{proof}

\section{Relation with a Poset of Edelman} \label{Section 6}\label{sec:edelman}

In~\cite{edelman}, Edelman introduced a subposet of the right weak order on the symmetric group $\mathfrak{S}_n$.
Although this poset is not a lattice, the intervals are known to each be distributive lattices, as is the case for the comb poset $\mathscr{C}_n$. %This map will turn out to be an isomorphism on each interval.
%In this section, we will show a deeper connection between the comb poset and $\mathscr{E}_n$: we will show that the natural map from $\mathscr{E}_n$ to $\mathscr{C}_{n + 1}$ is an order-preserving surjection, and go on to show that this map is in fact a lattice isomorphism on each of the intervals.

\begin{defn}
  The \textbf{right weak order on $\mathfrak{S}_n$} is a partial ordering of the elements of $\Sn$ defined as the transitive closure of the following covering relation: a permutation $\sigma$ covers a permutation $\tau$ if $\sigma$ is obtained from $\tau$ by a transposition of $\tau(i)$ and $\tau(i+1)$, two adjacent elements of the one line notation of $\tau$, such that $\tau(i)<\tau(i+1)$.  
  \end{defn}

Edelman imposed an additional constraint on this ordering, under which $\sigma$ covers $\tau$, if, after the transposition of $\tau(j)$ and $\tau({j + 1})$ as above, nothing to the left of $\tau({j + 1})$ in $\sigma$ is greater than $\tau({j + 1})$. This restriction results in a subposet of the right weak ordering on $\mathfrak{S}_n$. Denote this poset by $\mathscr{E}_n$.

\begin{example}
  Figure~\ref{fig: E3 hasse} depicts the Hasse diagram of $\mathscr{E}_3$, with an additional dashed edge indicating the extra order relation in the right weak order on $\mathfrak{S}_3$.
\end{example}

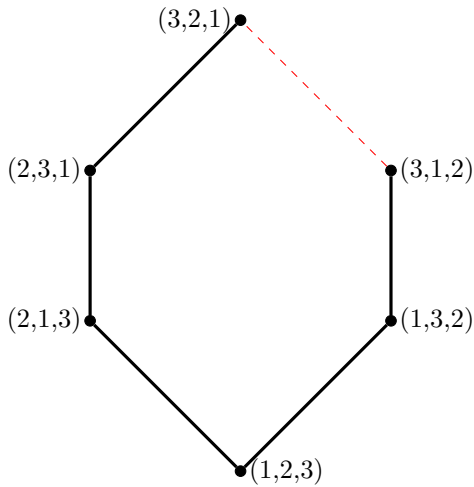
\begin{figure}[htbp]\begin{center}
\begin{tikzpicture}
    \tikzstyle{every node}=[draw,circle,fill=black,minimum size=4pt,
                            inner sep=0pt]

    \draw (0,0) node (01) [label=right:$(1\text{,}2\text{,}3)$]{};
    \draw (-2,2) node (02) [label=left:$(2\text{,}1\text{,}3)$]{};
    \draw (2,2) node (03) [label=right:$(1\text{,}3\text{,}2)$]{};
    \draw (-2,4) node (04) [label=left:$(2\text{,}3\text{,}1)$]{};
    \draw (2,4) node (05) [label=right:$(3\text{,}1\text{,}2)$]{};
    \draw (0,6) node (06) [label=left:$(3\text{,}2\text{,}1)$]{};
    \draw[very thick] (01) -- (02);
    \draw[very thick] (01) -- (03);
    \draw[very thick] (02) -- (04);
    \draw[very thick] (03) -- (05);
    \draw[very thick] (04) -- (06);
    \draw[dashed, red] (05) -- (06);
    
\end{tikzpicture}
\caption{Edelman's Poset $\mathscr{E}_3$.}
\label{fig: E3 hasse}
%\label{E3}
\end{center}
\end{figure}

%To define a map from $\mathscr{E}_n$ to $\mathscr{C}_{n+1}$, we will make use of the in-order labeling from Section~\ref{sec:meetjoin} to define a map from $\mathscr{E}_n$ to the set of pruned trees on $n$ vertices.

\begin{defn}\label{def: inverse stanley}
	The \textbf{pruned tree map}, $p \colon \mathfrak{S}_n \to \{\text{pruned trees on $n$ vertices}\}$, is defined recursively as follows. For $x \in \mathfrak{S}_1$, $p(x)$ is the tree with a single vertex. Then, for $n > 1$ and $x \in \mathfrak{S}_n$, define
	\[
		p(x)=
		\begin{tikzpicture}[baseline]
			\tikzstyle{every node}=[draw,circle,fill=black,minimum size=4pt,
                           inner sep=0pt]
			   \draw(0,1) node (01) {};
			   \draw (-1,0) node (02) [label=below:$p(x_<)$] {};
			   \draw (1,0) node (03) [label=below:$p(x_>)$] {};
			   \draw (01)--(02);
			   \draw (01)--(03);
		\end{tikzpicture}
	\]
	where $x_<=(x_{i_1},\ldots,x_{i_k})$ where $i_1<\cdots <i_k$ are the indices of all elements of $x$ less than $x_1$ and $x_>$ is defined similarly for elements of $x$ greater than $x_1$. Extend $p$ to a map $\beta \colon \mathfrak{S}_n \to \mathbb{T}_{n+1}$ by attaching leaves to $p(x)$ to give a binary tree (in other words, ``unpruning'' $p(x)$).
\end{defn}

\begin{remark}
  Amending the definition of $p$ slightly so that the root of $p(x)$ is labeled by $x_1$ results in the pruned tree having the \emph{in-order labeling}, where a vertex's label is greater than those of the vertices in its left subtree and smaller than those of the vertices in its right subtree. This labeled tree is, in fact, the unbalanced binary search tree for the permutation. (See~\cite{knuth}.) The pruned tree map is also related to the bijection between permutations and increasing binary trees on $n$ vertices (see~\cite[p.\ 24]{stanley}): the pruned tree associated to $w$ is the increasing binary tree associated to $w^{-1}$ with the labels removed. Consequently, the pruned tree map is a surjection.
\end{remark}

\begin{example}
	Figure~\ref{fig: bigedelmanex} shows $p\colon \mathfrak{S}_4\to \{\text{pruned trees with 4 vertices}\}$. Permutations having the same image are circled.
\end{example}
\begin{figure}[h]
    \begin{center}
	\includegraphics[scale=.5]{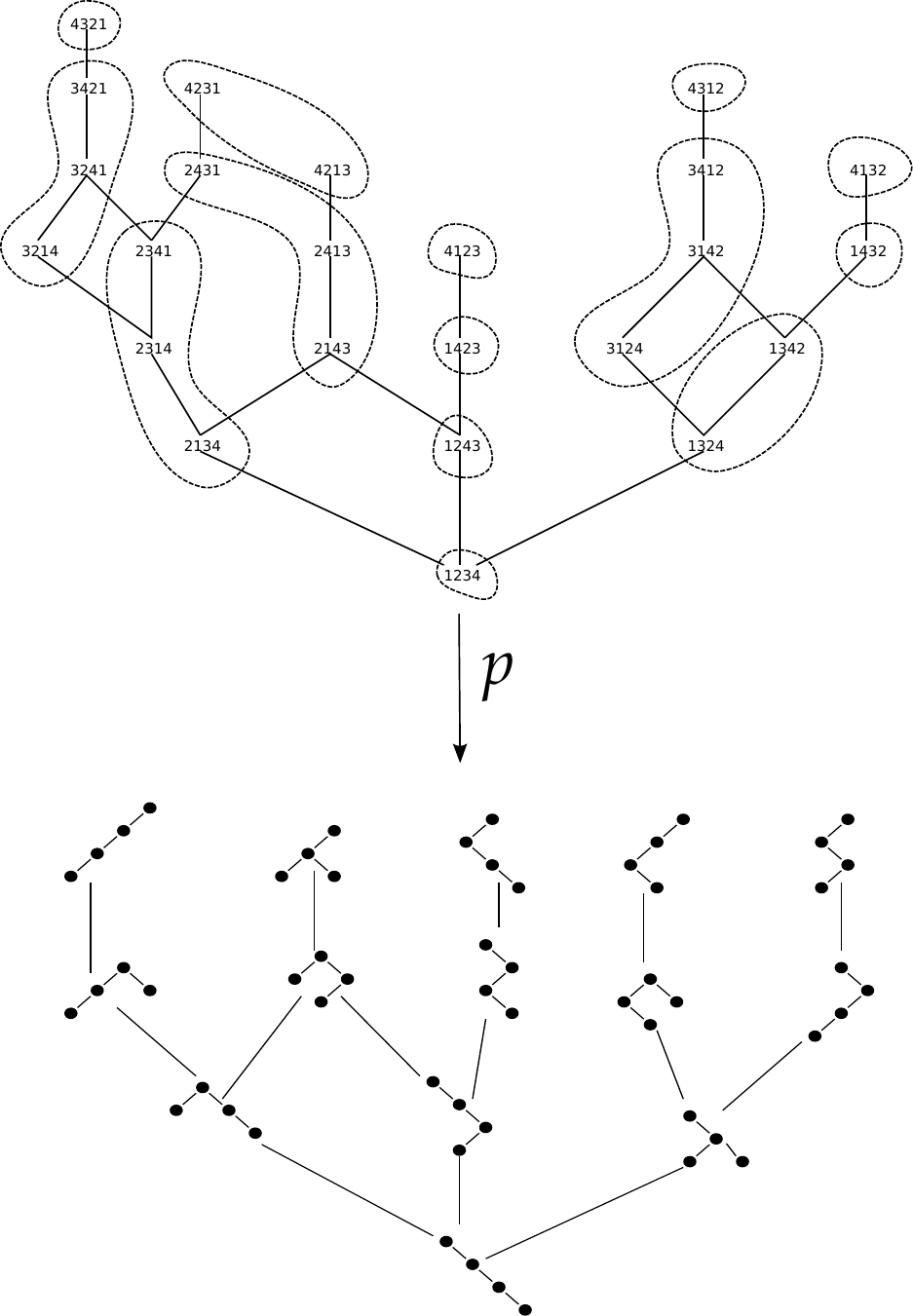}
	\caption{}
	\label{fig: bigedelmanex}
  \end{center}
\end{figure}

\begin{thm} \label{thm: inverse stanley is order preserving}
	The map $p \colon \mathfrak{S}_n \to \mathbb{T}_{n+1}$  gives an order-preserving surjection from $\mathscr{E}_n$ to $\mathscr{C}_{n + 1}$.
\end{thm}

\begin{proof}
  As noted above, it is well-known that $p$ is a surjection. It suffices to show that if $\sigma \lessdot \tau$ in $\mathscr{E}_n$ and $T_1=p(\sigma)$ and $T_2=p(\tau)$, then either $T_1 = T_2$, or $T_1 \lessdot T_2$ in $\mathscr{C}_{n + 1}$.

  Suppose 
  \[ 
  \sigma=(x_1, x_2, \ldots, x_j, x_{j + 1}, \ldots, x_n) \in \mathfrak{S}_n \quad\text{and} \quad \tau=(x_1, x_2, \ldots, x_{j + 1}, x_j, x_{j + 2}, \ldots, x_n) \in \mathfrak{S}_n,
  \]
  with $\tau$ covering $\sigma$ in $\mathscr{E}_n$. One then has that $x_s < x_{j + 1}$ for all $s < j+1$. Now, if $j = 1$, then the transposition changing $\sigma$ to $\tau$ corresponds, in the image of $p$, to a left rotation centered on the root, and therefore $T_1 \lessdot T_2$ in $\mathscr{C}_{n + 1}$. So assume $j \neq 1$; in other words, $x_j$ is not the root $x_1$ of the tree.

Recall that $x_s <x_{j+1}$ for all $s<j+1$. Suppose there is an $s<j$ such that $x_j <x_s<x_{j+1}$. Then, from the definition of $p$, one knows that the vertex labeled $x_j$ lies in the left subtrees of that labeled $x_s$ and $x_{j+1}$ lies in the right subtree. When $x_{j}$ and $x_{j+1}$ are exchanged to obtain $\tau$, their positions in the image of $p$ do not change and $T_1=T_2$.

Now suppose there is no $s<j$ such that $x_j <x_s<x_{j+1}$. In such a case, $T_1$ has the form 
\begin{center}

\begin{tikzpicture}
    \tikzstyle{every node}=[draw,circle,fill=white,minimum size=10pt,
                            inner sep=0pt]

    \draw (-1,3) node (01) [label=right:$S$]{};
    \draw (-1,1) node (02) [label=left:$X$]{};
    \draw (0,0) node (03) [label=left:$Y$]{};
    \draw (2,0) node (04) [label=right:$Z$]{};
    \node[draw=none] at (1.5, 2.5) {\Large $T_1$};
    
    \tikzstyle{every node}=[draw,circle,fill=black,minimum size=4pt,
                            inner sep=0pt]

    \draw (0,2) node (05) [label=right:$x_j$]{};
    \draw (1,1) node (06) [label=right:$x_{j + 1}$]{};
    \draw (01) -- (05);
    \draw (05) -- (02);
    \draw[thick, red] (05) -- (06);
    \draw (06) -- (03);
    \draw (06) -- (04);

\end{tikzpicture}
\end{center}
Here the white circle $S$ denotes the parent tree of the entire subtree shown, with the condition that $x_j$ and $x_{j + 1}$ lie on the right arm. The white circles $X$, $Y$ and $Z$ denote arbitrary subtrees, whose interpretations in terms of the elements in $\sigma$ are as follows: $X$ is the image under $P$ of the ordered sequence of elements appearing \emph{after} $x_j$ which are less than $x_j$, while $Z$ is the ordered sequence of elements appearing \emph{after} $x_{j + 1}$ which are greater than $x_{j + 1}$, and $Y$ is the ordered sequence of elements appearing after $x_j$ that lie between $x_j$ and $x_{j + 1}$.

Now, consider what happens to $T_2$, when $x_j$ and $x_{j + 1}$ are exchanged. The tree $T_2$ is depicted below.

\begin{center}

\begin{tikzpicture}
    \tikzstyle{every node}=[draw,circle,fill=white,minimum size=10pt,
                            inner sep=0pt]

    \draw (-1,3) node (01) [label=right:$S$]{};
    \draw (-2,0) node (02) [label=left:$X'$]{};
    \draw (0,0) node (03) [label=left:$Y'$]{};
    \draw (1,1) node (04) [label=right:$Z'$]{};
    \node[draw=none] at (-1.5, 2) {\Large $T_2$};
    
    \tikzstyle{every node}=[draw,circle,fill=black,minimum size=4pt,
                            inner sep=0pt]

    \draw (0,2) node (05) [label=right:$x_{j + 1}$]{};
    \draw (-1,1) node (06) [label=left:$x_j$]{};
    \draw (01) -- (05);
    \draw (05) -- (04);
    \draw[thick, red] (05) -- (06);
    \draw (06) -- (02);
    \draw (06) -- (03);

\end{tikzpicture}

\end{center}
Here, $S$ is going to be unchanged, and $x_j$ and $x_{j + 1}$ must move as shown. In addition, there will be subtrees $X'$, $Y'$, and $Z'$ as drawn above. However, notice that, if one considers what these subtrees must be with respect to the permutation $\tau$, the fact that $x_j$ and $x_{j + 1}$ are adjacent forces the conclusion that the subtrees are unchanged from $\sigma$, or in other words that $X = X'$, $Y = Y'$ and $Z = Z'$. So then, $T_2$ is obtained by a left rotation centered on a vertex on the right arm of $T_1$. Therefore, $T_2$ covers $T_1$ in $\mathscr{C}_{n + 1}$, completing the proof.
\end{proof}

To relate the intervals of $\mathscr{C}_{n+1}$ to those of $\mathscr{E}_n$ more deeply, a formal discussion of $\mathscr{E}_n$ is needed. In \cite{edelman}, Edelman defined the following order on the inversion set of a permutation $\sigma$.

\begin{defn} \label{defn: edelmaninv}
Define $I(\sigma) := \{(j, i) : j > i \text{ and } \sigma^{-1}(j) < \sigma^{-1}(i)\}$. Order $I(\sigma)$, with $(k, \ell) \geq (j, i)$ if and only if $k \geq j$ and $\sigma^{-1}(\ell) \leq \sigma^{-1}(i)$. In a slight abuse of notation, the poset $(I(\sigma), <)$ shall be referred to as $I(\sigma)$ as well.
\end{defn}

\begin{thm}[{Edelman, \cite[Theorem 2.13]{edelman}}] \label{thm: edelman}
  $[e,w]_{\mathscr{E}_n} \simeq J(I(w))$, where $[e,w]_{\mathscr{E}_n} = \{v \in \mathfrak{S}_n : v \leq_{\mathscr{E}_n} w\}$, via $v \mapsto I(v)$.
\end{thm}

\begin{defn} \label{defn: leftarms}
  Fix a permutation $w \in \mathfrak{S}_n$. Let $T_w$ be the image of $w$ under the pruned tree map, $p$. Recall the reduced pruned poset from Definition~\ref{def: rpposet}. Here it will be useful to label its vertices by the labels they have in $T_w$, rather than by pairs of parentheses as in the definition of $P_T$. Define a map $f : P_{T_w} \to I(w)$ as follows: $f(j) = (i, j)$, where $i$ is the smallest label of a vertex of $T_w$ such that $j$ lies in the left subtree of $i$.
\end{defn}

\begin{example}
	Suppose $w = (4, 9, 2, 1, 8, 3, 6, 7, 5) \in \mathfrak{S}_9$. Figure~\ref{fig: f example} depicts $T_w$, $P_{T_w}$ and $I(w)$, with the image of $f$ indicated in $I(w)$.
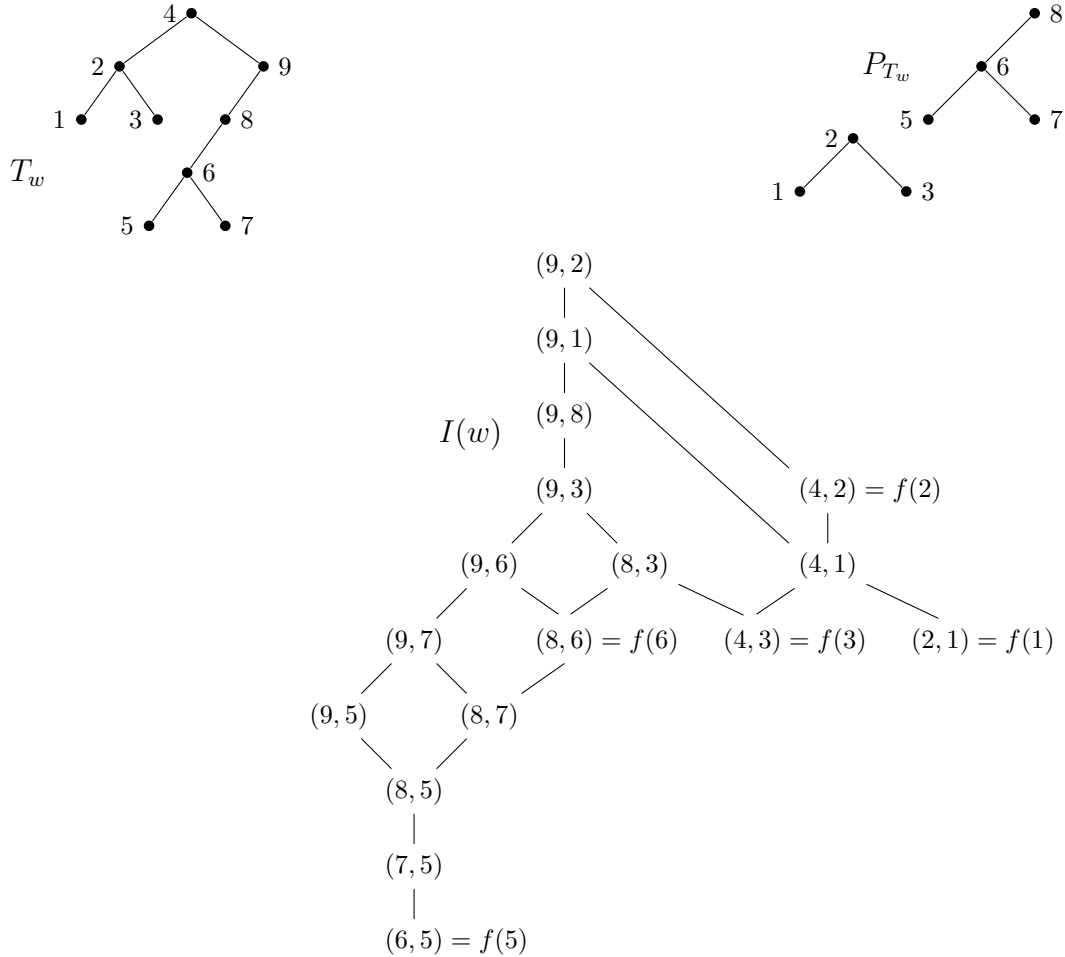
\begin{figure}	\begin{center}
	  \begin{tikzpicture}
      \matrix[]{
      \begin{scope}[vertex/.style={circle,fill=black,minimum size=4pt,inner sep=0mm}]
        \node[vertex,label=left:4] (4) at (0,0) {};
        \node[vertex,below left of=4,xshift=-.25cm,label=left:2] (2) {};
        \node[vertex,below right of=4,xshift=.25cm,label=right:9] (9) {};
        \node[vertex,below left of=2,xshift=.2cm,label=left:1] (1) {};
        \node[vertex,below right of=2,xshift=-.2cm,label=left:3] (3) {};
        \node[vertex,below left of=9,xshift=.2cm,label=right:8] (8) {};
        \node[vertex,below left of=8,xshift=.2cm,label=right:6] (6) {};
        \node[vertex,below left of=6,xshift=.2cm,label=left:5] (5) {};
        \node[vertex,below right of=6,xshift=-.2cm,label=right:7] (7) {};
        \node[below left of=1] {\Large{$T_w$}};
        \draw (4)--(2)--(1);
        \draw (2)--(3);
        \draw (4)--(9)--(8)--(6)--(5);
        \draw (6)--(7);
      \end{scope}
      &&[-4cm] 
      \begin{scope}[vertex/.style={circle,fill=black,minimum size=4pt,inner sep=0mm}]
      \node[vertex,label=right:8] (8) at (0,0) {};
      \node[vertex,below left of=8,label=right:6] (6) {};
      \node[vertex,below left of=6,label=left:5] (5) {};
      \node[vertex,below right of=6,label=right:7] (7) {};
      \node[vertex,left of=5,label=left:2,yshift=-.25cm] (2) {};
      \node[vertex,below left of=2,label=left:1] (1) {};
      \node[vertex,below right of=2,label=right:3] (3) {}; 
      \node[left of=6,xshift=-.25cm] {\Large{$P_{T_w}$}};
      \draw(8)--(6)--(5);
      \draw(6)--(7);
      \draw(2)--(1);
      \draw(2)--(3); 
      \end{scope}
      \\
      &
      \begin{scope}[every node/.style={anchor=west}]
        \node (65) at (0,0) {$(6,5)=f(5)$};
        \node (75) at (0,1) {$(7,5)$};
        \node (85) at (0,2) {$(8,5)$};
        \node (95) at (-1,3) {$(9,5)$};
        \node (87) at (1,3) {$(8,7)$};
        \node (97) at (0,4) {$(9,7)$};
        \node (86) at (2,4) {$(8,6)=f(6)$};
        \node (96) at (1,5) {$(9,6)$};
        \node (83) at (3,5) {$(8,3)$};
        \node (93) at (2,6) {$(9,3)$};
        \node (43) at (4.5,4) {$(4,3)=f(3)$};
        \node (21) at (7,4) {$(2,1)=f(1)$};
        \node (41) at (5.5,5) {$(4,1)$};
        \node (42) at (5.5,6) {$(4,2)=f(2)$};
        \node (98) at (2,7) {$(9,8)$};
        \node (91) at (2,8) {$(9,1)$};
        \node (92) at (2,9) {$(9,2)$};
        \node (a) at (4.5,5) {\phantom{(4,3)}};
        \node (b) at (7,4) {\phantom{(2,1)}};
        \node[left of=98,xshift=-.25cm,yshift=-.25cm] {\Large{$I(w)$}}; 
        \draw (65.north -| 75)--(75)--(85)--(95);
        \draw (85)--(87)--(97);
        \draw (95)--(97)--(96)--(93);
        \draw (87)--(86.south -| 93);
        \draw[shorten <=1mm] (86.north -| 93)--(96);
        \draw[shorten <=1mm] (86.north -| 93)--(83);
        \draw (83)--(93)--(98)--(91)--(92);
        \draw[shorten <=1mm] (43.north -| a)--(41);
        \draw (41)--(42.south -|41);
        \draw[shorten <=1mm] (43.north -| a)--(83);
        \draw (41.north west)--(91);
        \draw (42.north west)--(92);
        \draw (21.north -| b)--(41);
      \end{scope}
      &
      \\
      };
    \end{tikzpicture}
	\end{center}
	\caption{$T_w$, $P_{T_w}$ and $I(w)$ with the image of $f$ indicated.}
	\label{fig: f example}
\end{figure}
\end{example}

\begin{prop} \label{prop: forderpres}
The map $f$ is order-preserving.
\end{prop}

\begin{proof}
  It suffices to show that if $j > k$, and $j$ covers $k$ in $P_{T_w}$, then $f(j) > f(k)$. Since $j$ covers $k$, one has that $k$ is a child of $j$, and there are two cases.
\begin{enumerate}
\item If $k$ is a left child of $j$, then $f(k) = (j, k)$. By the definition of the pruned tree map (Definition \ref{def: inverse stanley}), one knows $w^{-1}(j) < w^{-1}(k)$. Suppose $f(j) = (i, j)$. Then, by definition, $j < i$, which means that $(j, k) < (i, j)$ in $I(w)$, as desired.

\item If $k$ is a right child of $j$, then $f(j) = (i, j)$ means that $f(k) = (i, k)$. Now, $w^{-1}(j) < w^{-1}(k)$, and so $(i, j) > (i, k)$, as desired.

\end{enumerate}
These cover all the cases, proving the result.
\end{proof}

\begin{defn}\label{def: BP}
  Let $P_1,P_2$ be two posets and suppose $\phi \colon P_1\to P_2$ is order-preserving. Then $\phi$ induces a map $J(\phi)\colon J(P_2)\to J(P_1)$ defined by $J(\phi)(I)=\phi^{-1}(I)$. One calls $J(\phi)$ the \emph{Birkhoff-Priestley dual} to $\phi$. In fact, $J(\phi)$ is a lattice morphism. 
\end{defn}

For further details on Birkhoff-Priestley duality, see~\cite[Theorem 3.4.1]{stanley}.

\begin{comment}
\begin{remark}
  The ensuing discussion will make use of the following observation. Given an order ideal in $P_{T_w}$, associated to it is an in-order labeled pruned tree.

\begin{itemize}

\item Each connected component of the Hasse diagram will be the left subtree of some element on the right arm of the pruned tree.

\item Any elements of $[n]$ not in $P_T$ are arranged as the right arm of the tree, in descending order.

\item Given a connected component whose vertices are labeled $\{i_i, \cdots, i_k\}$, it is the left subtree of $\max\{i_j\} + 1$, where one takes the maximum in numerical order, and not the $
  P_{T_w}$ order.

\end{itemize}
\end{remark}
\end{comment}

\begin{thm} \label{thm: B-Pdual}
For each $w \in \mathfrak{S}_n$, the map $f \colon P_{T_w}\to I(w)$ defined in Definition~\ref{defn: leftarms} is Birkhoff-Priestley dual to the pruned tree map $p \colon [e,w]_{\mathscr{E}_n} \to J(P_{T_w})$. In particular, $p \colon \mathscr{E}_n \to \mathscr{C}_n$ becomes a lattice morphism when restricted to any interval in $\mathscr{E}_n$. As a commutative diagram, one has
\[
	\begin{tikzpicture}[
		bij/.style={above,sloped,inner sep=0.5pt}]
		\matrix (a) [matrix of math nodes,row sep=3em,column sep=3em]
		{
		\mathscr{E}_n & \mathscr{C}_{n+1} \\
		{[e,w]}_{\mathscr{E}_n} & {[\RCT{n+1},T_w]}_{\mathscr{C}_{n+1}} \\
		J(I(w)) & J(P_{T_w}) \\
		};
		\path[->] (a-1-1) edge node[auto] {$p$} (a-1-2);
		\path[->] (a-2-1) edge node[auto] {$p|_{[e,w]_{\mathscr{E}_n}}$}(a-2-2);
		\path[right hook->] (a-2-1) edge (a-1-1);
		\path[right hook->] (a-2-2) edge (a-1-2);
		\path[->] (a-3-1) edge node[auto] {$J(f)$}(a-3-2);
		\path[->] (a-3-1) edge node[bij] {$\sim$}(a-2-1);
		\path[->] (a-3-2) edge node[bij] {$\sim$}(a-2-2);
	\end{tikzpicture}
\]
\end{thm}

\begin{example}
	Figure~\ref{fig: birkhoffexample} depicts Theorem~\ref{thm: B-Pdual} on the interval $[e,4213]_{\mathscr{E}_n}$.
	\begin{figure}		\includegraphics[scale=.5]{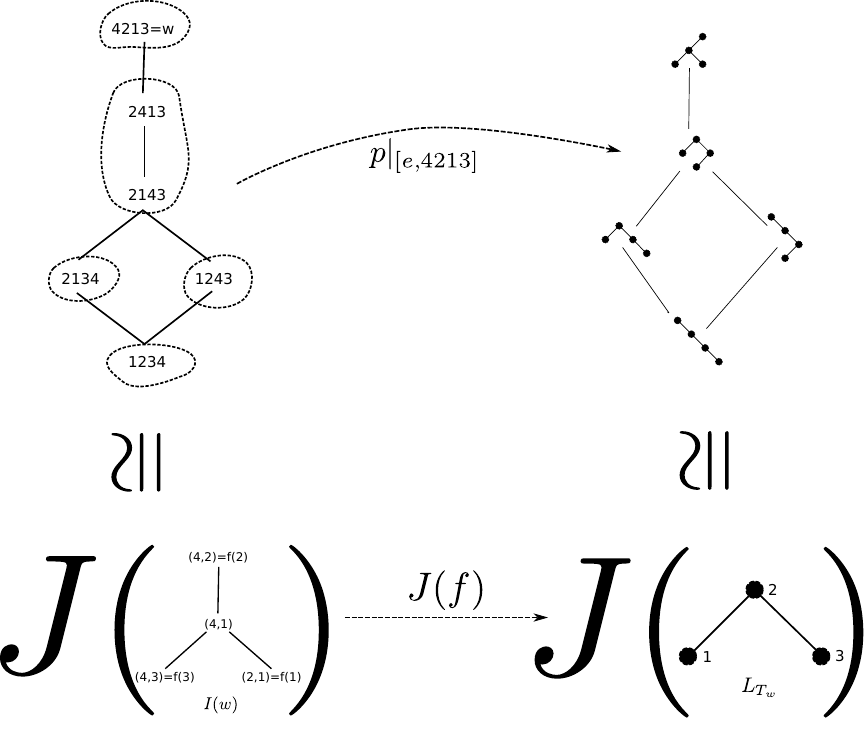}
		\caption{$p|_{[e,4213]}$ and $J(f)$ for the interval $[e,4213]_{\mathscr{E}_n}$.}\label{fig: birkhoffexample}
	\end{figure}
\end{example}

\begin{proof}
Begin by noting that, strictly speaking, the Birkhoff-Priestley dual to $f$, $J(f)$, is not a map from $[e,w]_{\mathscr{E}_n} \to J(P_{T_w})$ as $p$ is, but $J(f)\colon J(I(w))\to J(P_{T_w})$. However, from Theorem~\ref{thm: edelman}, $J(I(w))\simeq [e,w]_{\mathscr{E}_n} $, so one can use $p$ in place of such a $J(f)$.

Fix $w \in \mathfrak{S}_n$. From Theorem~\ref{thm: inverse stanley is order preserving} one has that $p \colon [e,w]_{\mathscr{E}_n} \to J(P_{T_w})$ is order-preserving. Then one must show that $p(I)$ is, in fact, $f^{-1}(I)$. 

Induct on the number of inversions in a permutation in $[e,w]_{\mathscr{E}_n}\simeq J(I(w))$. Note that the claim is trivially true for $(1, 2, \ldots, n)$, the identity permutation, which corresponds to $\varnothing \in J(I(w))$.

Now consider the permutation $\tau \in [e,w]_{\mathscr{E}_n} $, and $f^{-1}(I(\tau)) = T_\tau$. Suppose $\sigma$ covers $\tau$. Then, $\sigma$ has precisely one more inversion than $\tau$; call this inversion $(i, j)$, with $j<i$. 

Suppose there is an inversion $(\ell, j)$ in both $\tau$ and $\sigma$ with $\tau^{-1}(\ell) = \sigma^{-1}(\ell) < \sigma^{-1}(i)$. Then, $j$ is in the left subtree of $\ell$ in $T_\tau$ and $T_\sigma$, meaning that $j$ is not the parent of $i$ in $T_\tau$ and so adding the inversion $(i, j)$ does not change $T_\tau$, forcing $T_\tau = T_\sigma$, as desired. Thus, one may concentrate on the case where there is no such inversion $(\ell,j)$, so $j$ is a left-right maximum in $\tau$.

There are two cases:

\begin{enumerate}
\item If $(i, j)$ is not in the image of $f$, then $f^{-1}(I(\sigma)) = f^{-1}(I(\tau))$, and so one must show that $T_\sigma = T_\tau$. One knows $i$ and $j$ are adjacent in $\tau$, and $i$ is a left-right maximum. In particular, this means that neither $\tau$ nor $\sigma$ has an inversion $(k, i)$, meaning $i$ lies in the right arm of both $T_\tau$ and $T_\sigma$. Recalling that $j$ is a left-right maximum in $\tau$ one has 

\begin{center}

\begin{tikzpicture}
    \tikzstyle{every node}=[draw,circle,fill=white,minimum size=10pt,
                            inner sep=0pt]

    \draw (-4,1) node (02) [label=below:$B$]{};
    \draw (-3,0) node (03) [label=below:$C$]{};
    \draw (-1,0) node (04) [label=below:$D$]{};
    \draw (-4,3) node (07) [label=above:$A$]{};
    \node[draw=none, fill=white] at (-2, 2) {\Large $T_\tau$};
    
    \tikzstyle{every node}=[draw,circle,fill=black,minimum size=4pt,
                            inner sep=0pt]

    \draw (-3,2) node (05) [label=right:$j$]{};
    \draw (-2,1) node (06) [label=right:$i$]{};
    \draw (05) -- (02);
    \draw (05) -- (06);
    \draw (06) -- (03);
    \draw (06) -- (04);
    \draw (05) -- (07);

    \tikzstyle{every node}=[draw,circle,fill=white,minimum size=10pt,
                            inner sep=0pt]

    \draw (1,0) node (12) [label=below:$B$]{};
    \draw (3,0) node (13) [label=below:$C$]{};
    \draw (4,1) node (14) [label=below:$D$]{};
    \draw (2,3) node (17) [label=above:$A$]{};
    \node[draw=none, fill=white] at (4, 2) {\Large $T_\sigma$};
    
    \tikzstyle{every node}=[draw,circle,fill=black,minimum size=4pt,
                            inner sep=0pt]

    \draw (3,2) node (15) [label=right:$i$]{};
    \draw (2,1) node (16) [label=left:$j$]{};
    \draw (15) -- (14);
    \draw (15) -- (16);
    \draw (16) -- (12);
    \draw (16) -- (13);
    \draw (15) -- (17);
    
\end{tikzpicture}
\end{center}
Since $(i,j)$ is not in the image of $f$, one cannot have that $j$ is the left child of $i$ in $T_w$. However, $j$ is the left child of $i$ in $T_\sigma$ and $\sigma < w$, meaning $j$ must also be the left child of $i$ in $T_w$, a contradiction. Thus $(i,j)$ must be in the image of $f$.

\item In the second case, suppose $(i, j)$ is in the image of $f$. Then, the addition of the inversion of $(i, j)$ to $\tau$ results in $\sigma$, and in the following rotation from $T_\tau$ to $T_\sigma$.

\begin{center}

\begin{tikzpicture}
    \tikzstyle{every node}=[draw,circle,fill=white,minimum size=10pt,
                            inner sep=0pt]

    \draw (-4,1) node (02) [label=below:$B$]{};
    \draw (-3,0) node (03) [label=below:$C$]{};
    \draw (-1,0) node (04) [label=below:$D$]{};
    \draw (-4,3) node (07) [label=above:$A$]{};
    \node[draw=none, fill=white] at (-2, 2) {\Large $T_\tau$};
    
    \tikzstyle{every node}=[draw,circle,fill=black,minimum size=4pt,
                            inner sep=0pt]

    \draw (-3,2) node (05) [label=right:$j$]{};
    \draw (-2,1) node (06) [label=right:$i$]{};
    \draw (05) -- (02);
    \draw (05) -- (06);
    \draw (06) -- (03);
    \draw (06) -- (04);
    \draw (05) -- (07);

\path[draw,thick,black,->] (-1,2) -- (1,2);

    \tikzstyle{every node}=[draw,circle,fill=white,minimum size=10pt,
                            inner sep=0pt]

    \draw (1,0) node (12) [label=below:$B$]{};
    \draw (3,0) node (13) [label=below:$C$]{};
    \draw (4,1) node (14) [label=below:$D$]{};
    \draw (2,3) node (17) [label=above:$A$]{};
    \node[draw=none, fill=white] at (4, 2) {\Large $T_\sigma$};
    
    \tikzstyle{every node}=[draw,circle,fill=black,minimum size=4pt,
                            inner sep=0pt]

    \draw (3,2) node (15) [label=right:$i$]{};
    \draw (2,1) node (16) [label=left:$j$]{};
    \draw (15) -- (14);
    \draw (15) -- (16);
    \draw (16) -- (12);
    \draw (16) -- (13);
    \draw (15) -- (17);
    
\end{tikzpicture}

\end{center}
One needs to show that $f^{-1}(I(\sigma))$ is the order ideal in $P_{T_w}$ that corresponds to $T_\sigma$. Now, since $T_\tau$ and $T_\sigma$ differ only in this rotation, one need only show that

\begin{center}
  \begin{tikzpicture}
	\tikzstyle{every node}=[draw,circle,fill=white,minimum size=10pt,
                            inner sep=0pt]
		\draw (0,0) node (01) [label=below:$B$]{};
		\draw (2,0) node (02) [label=below:$C$]{};
	 \tikzstyle{every node}=[draw,circle,fill=black,minimum size=4pt,
                            inner sep=0pt]
		\draw (1,1) node (03) [label=right:$j$]{};

		\draw (01)--(03);
		\draw (02)--(03);

  \end{tikzpicture}
\end{center}
appears in $P_{T_w}$. Left-right maxima occur only on the right arm of the image of a permutation under the pruned tree map, and so subsequent inversions on the way from $\sigma$ to $w$ result in rotations in the pruned tree that cannot affect the children of $j$. Hence, the above subtree
%
\iffalse
\begin{center}
  \begin{tikzpicture}
	\tikzstyle{every node}=[draw,circle,fill=white,minimum size=10pt,
                            inner sep=0pt]
		\draw (0,0) node (01) [label=below:$B$]{};
		\draw (2,0) node (02) [label=below:$C$]{};
	 \tikzstyle{every node}=[draw,circle,fill=black,minimum size=4pt,
                            inner sep=0pt]
		\draw (1,1) node (03) [label=right:$j$]{};

		\draw (01)--(03);
		\draw (02)--(03);

  \end{tikzpicture}
\end{center}
\fi
appears in $P_{T_w}$, and so, $T_\sigma$ is the pruned tree associated to $I(\sigma)$.
\end{enumerate}
These two cases cover all possibilities, concluding the proof.
\end{proof}

\section{The ParseWords Function for the Comb Poset} \label{Section 7}

The number of common parsewords for any two trees having a common upper bound in $\mathscr{C}_n$ can be computed precisely. Recall that $w \in \PW(T)$ means that $T$ admits a labeling of its vertices by $0,1,2$ such that the leaves are labeled by the word $w$, the children of each vertex have distinct labels and no vertex has the same label as either of its children. Recent work by Cooper, Rowland and Zeilberger in~\cite{crz} led us to first consider the comb poset. They showed that a statement equivalent to the Four Color Theorem due to Kaufmann in~\cite{kauffman} is, in turn, equivalent to $\PW(T_1,T_2)\ne \emptyset$ for all $T_1,T_2 \in \mathbb{T}_n$ for any $n \in \mathbb{N}$.
%Kauffman showed in~\cite{kauffman} that the Four Color Theorem is equivalent to $\PW(T_1,T_2)\ne \emptyset$ for all $T_1,T_2 \in \mathbb{T}_n$ for any $n \in \mathbb{N}$.  It was further recent work by Cooper, Rowland and Zeilberger in~\cite{crz} that led us to first consider the comb poset.

\begin{example}
An example of a tree parsing the word $2202$ is shown in Figure~\ref{PW ex}.
\begin{figure}[htpb]
\begin{center}
\begin{tikzpicture}
    \tikzstyle{every node}=[draw,circle,fill=black,minimum size=4pt,
                            inner sep=0pt]

    \draw (-3,1.5) node (L12) [label=below:$2$]{};
    \draw (-1.5,3) node (V0) [label=left:$0$]{};
    \draw (0,4.5) node (R1) [label=above:$1$]{};
    \draw (1.5,3) node (L42) [label=below:$2$]{};
    \draw (0,1.5) node (V1) [label=right:$1$]{};
    \draw (-1.5,0) node (L22) [label=below:$2$]{};
    \draw (1.5,0) node (L30) [label=below:$0$]{};
    \draw[thick,blue] (R1) -- (L42);
    \draw[thick,blue] (R1) -- (V0);
    \draw[thick,blue] (V0) -- (L12);
    \draw[thick,blue] (V0) -- (V1);
    \draw[thick,blue] (V1) -- (L22);
    \draw[thick,blue] (V1) -- (L30);
    
\end{tikzpicture}
\end{center}
\caption{}
\label{PW ex}
\end{figure}
\end{example}

\begin{example}
An example of two trees parsing the same word $010$ is shown in Figure~\ref{CPW ex}.
\begin{figure}[htpb]
  \begin{center}

\begin{tikzpicture}

    \tikzstyle{every node}=[draw,circle,fill=black,minimum size=4pt,
                            inner sep=0pt]

    \draw (-4,1) node (02) [label=below:$0$]{};
    \draw (-3,0) node (03) [label=below:$1$]{};
    \draw (-1,0) node (04) [label=below:$0$]{};
    \draw (-3,2) node (05) [label=above:$1$]{};
    \draw (-2,1) node (06) [label=right:$2$]{};
    \node[draw=none, fill=white] at (-2, 2) {\Large $T_1$};
    
    \draw (05)[thick,blue] -- (02);
    \draw (05)[thick,blue] -- (06);
    \draw (06)[thick,blue] -- (03);
    \draw (06)[thick,blue] -- (04);

    \draw (1,0) node (12) [label=below:$0$]{};
    \draw (3,0) node (13) [label=below:$1$]{};
    \draw (4,1) node (14) [label=below:$0$]{};
    \draw (3,2) node (15) [label=above:$1$]{};
    \draw (2,1) node (16) [label=left:$2$]{};
    \node[draw=none, fill=white] at (2, 2) {\Large $T_2$};
    
    \draw (15)[thick,blue] -- (14);
    \draw (15)[thick,blue] -- (16);
    \draw (16)[thick,blue] -- (12);
    \draw (16)[thick,blue] -- (13);
    
    \path[draw,thick,red,<->] (-1,1) -- (1,1);

\end{tikzpicture}
\end{center}
\caption{}
\label{CPW ex}
\end{figure}
\end{example}

\begin{example}
  The common parsewords for the trees in Figure~\ref{CPW ex} are 101, 202, 010, 212, 020, 121.
\end{example}

To simplify notation, let $T_{\leq b}$ be the subtree of a tree $T$ having the vertex $b$ as its root. 

\begin{prop}[{Common root property, \cite[Proposition 2]{crz}}] \label{prop: common root}
  If two trees $T_1, T_2 \in \mathbb{T}_n$ parse the same word, then their roots receive the same label when the trees are labeled with a common parseword. Hence, if for $T_1, T_2 \in \mathbb{T}_n$, there are vertices $b_i$ in $T_1$ and $b_j$ in $T_2$ such that $T_{1_{\leq b_i}}$ and $T_{2_{\leq b_j}}$ have precisely the same leaves (i.e. both the dangling subtrees contain precisely the leaves $m_1$ through $m_2$, for some natural numbers $m_1 < m_2 \leq n$), then $b_i$ and $b_j$ receive the same label if one labels the trees with a common parse word.
\end{prop}

%\begin{proof}
%  The first assertion is a rephrasing of Proposition 2 in~\cite{crz}, which states that when a tree parsing $w$ is fully labeled using $w$, the label received by the root occurs in $w$ with parity different to the other two letters, as one can verify by induction on $n$. The second assertion can be seen by considering the dangling subtrees as trees in themselves.
%\end{proof}

If a tree $T$ parses a word $w$ and $X$ is a subtree of $T$, let $w(X)$ be the label received by the root of $X$ parsing $w$ and let $w_X$ be the segment of $w$ parsed by the subtree $X$.

\begin{defn}\label{def: leaf reduction}
	Say $T \in \mathbb{T}_n$ has a \emph{leaf reduction} at $(i,i+1)$ for $i \in \{1,\ldots,n-1\}$ if the leaves $i,i+1$ have a common parent:
			\begin{center}
			 \begin{tikzpicture}
	\tikzstyle{every node}=[draw,circle,fill=white,minimum size=20pt,
                            inner sep=0pt]
		\draw (-1.5,1) node (03) {$\hat{T}$};
	 \tikzstyle{every node}=[draw,circle,fill=black,minimum size=4pt,
                            inner sep=0pt]
		\draw (-2.25,0) node (02) [label=left:$\ell_i$]{};
		\draw (-0.75,0) node (01) [label=right:$\ell_{i + 1}$]{};

		\draw (01)--(03);
		\draw (02)--(03);

  \end{tikzpicture}
  \end{center}	
\begin{comment}
			  %\includegraphics{leafred1.png}
			\end{center}
		\item the leaf $\ell_1$ is attached to the root $r$ if $i=0$:
			\begin{center}
			 \begin{tikzpicture}
	\tikzstyle{every node}=[draw,circle,fill=white,minimum size=20pt,
                            inner sep=0pt]
		\draw (2.25,0) node (02) {$\hat{T}$};
	 \tikzstyle{every node}=[draw,circle,fill=black,minimum size=4pt,
                            inner sep=0pt]
		\draw (1.5,1) node (03) [label=above:$r$]{};
		\draw (1,0.5) node (01) [label=below:$\ell_1$]{};

		\draw (01)--(03);
		\draw (02)--(03);

  \end{tikzpicture}
	
			  %\includegraphics{leafred2.png}
			\end{center}
		\item the leaf $\ell_n$ is attached to the root $r$ if $i=n$:
			\begin{center}
			 \begin{tikzpicture}
	\tikzstyle{every node}=[draw,circle,fill=white,minimum size=20pt,
                            inner sep=0pt]
		\draw (-2.25,0) node (02) {$\hat{T}$};
	 \tikzstyle{every node}=[draw,circle,fill=black,minimum size=4pt,
                            inner sep=0pt]
		\draw (-1.5,1) node (03) [label=above:$r$]{};
		\draw (-1,0.5) node (01) [label=below:$\ell_n$]{};

		\draw (01)--(03);
		\draw (02)--(03);

  \end{tikzpicture}	
			  %\includegraphics{leafred3.png}
			\end{center}
	\end{enumrom}
\end{comment}
	Define $\hat{T}$ as in the above diagram, i.e.\ remove $\ell_i$ and $\ell_{i+1}$ from $T$. For $i \in \{1,\ldots, n-1\}$, define two maps $f^<_{i},f^>_i\colon \PW(\hat{T})\to \PW(T)$ sending $\hat{w}$ to the $w$ in $\PW(T)$ that uniquely extends $\hat{w}$ in such a way that $w_i<w_{i+1}$ or $w_i>w_{i+1}$, respectively.
	%If $i=0$, take $w_{i+1}$ to be $w(\hat{T})$ and if $u=n$, take $w_i$ to be $w(\hat{T})$.)
\end{defn}

\begin{prop}\label{prop: leaf collapse}
If $\{T_1\ldots,T_m\}\subset \mathbb{T}_n$ share a leaf reduction at $(i,i+1)$, then
\[
\PW(T_1,\ldots,T_m)=f_i^<(\PW(\hat{T}_1,\ldots,\hat{T}_m))\sqcup f_i^>(\PW(\hat{T}_1,\ldots,\hat{T}_m)).	
\]
\end{prop}
Proposition~\ref{prop: leaf collapse} is most frequently used several times in succession, to ``collapse'' a subtree common to two or move trees. In particular, it often allows a reduction to the case $T_1\wedge T_2\meet \cdots \meet T_m=\RCT{n}$.

\begin{cor} \label{prop: constant PW(T)}
	For $T \in \mathbb{T}_n$, one has $|\PW(T)|=3\cdot 2^{n-1}$.
\end{cor}

\begin{proof}
	The case $n=1$ is trivial. One can then induct on $n$, applying Proposition~\ref{prop: leaf collapse} to $T$ alone.
\end{proof}

\begin{prop}\label{prop: cover parse word}\label{unwritten}
	For any $T_1,T_2 \in \mathbb{T}_n$ differing by a single rotation (not necessarily a right arm rotation), 
	\begin{align*}
	  \PW(T_1,T_2)&=\{w \in \PW(T_1) \colon w(X) \ne w(Y)\} \\
	  &=\{w \in \PW(T_2) \colon w(Y) \ne w(Z)\} \\
	  &=\{w \in \PW(T_1) \colon w(X)=w(Z)\},
      \end{align*}
	where $X,Y,Z$ are subtrees as indicated below. 
	Furthermore, $|\PW(T_1,T_2)|=3\cdot 2^{n-2}$.
      \end{prop}

\begin{proof}
The conditions on $w$ in the first part of the claim can be checked by inspection. For the second part, the case $n=3$ can be checked directly. Taking this as a base case, one can induct on $n$. A rotation looks like
\begin{center}

\begin{tikzpicture}
    \tikzstyle{every node}=[draw,circle,fill=white,minimum size=10pt,
                            inner sep=0pt]

    \draw (-4,1) node (02) [label=below:$X$]{};
    \draw (-3,0) node (03) [label=below:$Y$]{};
    \draw (-1,0) node (04) [label=below:$Z$]{};
    \draw (-4,3) node (07) [label=above:$S$]{};
    \node[draw=none, fill=white] at (-2, 2) {\Large $T_1$};
    
    \tikzstyle{every node}=[draw,circle,fill=black,minimum size=4pt,
                            inner sep=0pt]

    \draw (-3,2) node (05) {};
    \draw (-2,1) node (06) {};
    \draw (05) -- (02);
    \draw[red,thick] (05) -- (06);
    \draw (06) -- (03);
    \draw (06) -- (04);
    \draw (05) -- (07);

    \tikzstyle{every node}=[draw,circle,fill=white,minimum size=10pt,
                            inner sep=0pt]

    \draw (1,0) node (12) [label=below:$X$]{};
    \draw (3,0) node (13) [label=below:$Y$]{};
    \draw (4,1) node (14) [label=below:$Z$]{};
    \draw (2,3) node (17) [label=above:$S$]{};
    \node[draw=none, fill=white] at (4, 2) {\Large $T_2$};
    
    \tikzstyle{every node}=[draw,circle,fill=black,minimum size=4pt,
                            inner sep=0pt]

    \draw (3,2) node (15) {};
    \draw (2,1) node (16) {};
    \draw (15) -- (14);
    \draw[red,thick] (15) -- (16);
    \draw (16) -- (12);
    \draw (16) -- (13);
    \draw (15) -- (17);
    
    \path[draw,thick,->] (-1,1.5) -- (1,1.5);

\end{tikzpicture}
\end{center}
Applying Proposition~\ref{prop: leaf collapse} to any of the subtrees $X,Y,Z$ not consisting of a single leaf, or to the subtrees taken together as a single subtree if all three are leaves, allows one to invoke the result for a smaller $n$ and obtain the desired result.
\end{proof}

\begin{thm}\label{thm: PWchain}
	Suppose $T_1 < T< T_2$ in $\mathscr{C}_n$. Then $\PW(T_1,T_2)=\PW(T_1,T_2,T)$.
\end{thm}

\begin{proof}
	It suffices to prove the statement for $T_1 \lessdot T < T_2$ in $\mathscr{C}_n$. Assume the theorem holds in this case and obtain the general case by induction on the length of a chain between $T_1$ and $T$. 
	
	Suppose one has $T_1 \lessdot T_1'\lessdot T_2'\lessdot \cdots \lessdot T_k'\lessdot T <T_2$. Then, by induction, $\PW(T_1,T_k',T_2)=\PW(T_1,T_2)$. Suppose $w \in \PW(T_1,T_2)=\PW(T_1,T_k',T_2)$. Furthermore, $w \in \PW(T_k',T_2)=\PW(T_k',T,T_2)$, so $w \in \PW(T_1,T,T_2)$ as desired. By definition, $\PW(T_1,T,T_2)\subset \PW(T_1,T_2)$, so $\PW(T_1,T,T_2)=\PW(T_1,T_2)$, as desired.

	To prove the initial case, now suppose $T_1 \lessdot T<T_2$. One has a sequence of right-arm rotations
	\begin{center}
	\begin{tikzpicture}
\tikzstyle{every node}=[draw,circle,fill=black,minimum size=4pt,
                            inner sep=0pt]

    \draw (-2.5,0.5) node (06) {};
    \draw (-3,1) node (04) {};
    
    \tikzstyle{every node}=[draw,circle,fill=white,minimum size=10pt,
                            inner sep=0pt]

    \draw (-3.5,1.5) node (02) [label=above:$S$]{};
    \draw (-3.5,0.5) node (03) [label=below:$X$]{};
    \draw (-3,0) node (05) [label=below:$Y$]{};
    \draw (-2,0) node (08) [label=below:$Z$]{};
    
    \draw[blue] (03) -- (04);
    \draw[blue] (05) -- (06);
    \draw[blue] (02) -- (04);
    \draw[blue] (04) -- (06);
    \draw[blue] (06) -- (08);
    
    \path[draw,thick,->] (-2,1) -- (-1,1);
    
    \tikzstyle{every node}=[draw,circle,fill=black,minimum size=4pt,
                            inner sep=0pt]
    \draw (0.25,0.5) node (25) {};
    \draw (0.75,1) node (26) {};
    
    \tikzstyle{every node}=[draw,circle,fill=white,minimum size=10pt,
                            inner sep=0pt]
\draw (0.25,1.5) node (24) [label=above:$S$]{};
\draw (-0.25,0) node (27) [label=below:$X$]{};
    \draw (1.25,0.5) node (28) [label=below:$Z$]{};
    \draw (0.75,0) node (29) [label=below:$Y$]{};
    
    \draw[blue] (25) -- (26);
    \draw[blue] (25) -- (27);
    \draw[blue] (24) -- (26);
    \draw[blue] (26) -- (28);
    \draw[blue] (25) -- (29);
    
    \path[draw,thick,->] (1.5,1) -- (2.5,1);
    
    \tikzstyle{every node}=[draw=none,circle,fill=none,minimum size=0pt,
                            inner sep=0pt]

    \draw (-3,-1) node (T1) {\Large$T_1$};
    \draw (1,-1) node (T2) {\Large$T$};

    \end{tikzpicture}
    \begin{tikzpicture}
    \tikzstyle{every node}=[draw=none,circle,fill=none,minimum size=0pt,
                            inner sep=0pt]
    \draw (1.25,1) node (28) {\Huge$\cdots$};
    \draw (1.25,-0.5) node (29) {};
    \draw (0.3,-0.5) node (30) {};
    \tikzstyle{every node}=[draw=none,circle,fill=none,minimum size=0pt,
                            inner sep=0pt]

    \draw (0,-1.3) node (T1) {};
    
    \end{tikzpicture}
    \begin{tikzpicture}
    \path[draw,thick,->] (1.5,1) -- (2.5,1);
    \draw (1.25,-0.5) node (29) {};
    \tikzstyle{every node}=[draw=none,circle,fill=none,minimum size=0pt,
                            inner sep=0pt]

    \draw (1,-1.3) node (T1) {};
    
    \end{tikzpicture}
    \begin{tikzpicture}
    \tikzstyle{every node}=[draw,circle,fill=black,minimum size=4pt,
                            inner sep=0pt]
    \draw (3.5,1) node (34) {};
    
    \tikzstyle{every node}=[draw,circle,fill=white,minimum size=10pt,
                            inner sep=0pt]
                            
    \draw (4,0.5) node (36) [label=below:$Y$]{};
    \draw (4,1.5) node (37) [label=above:$S$]{};
    \draw (3,0.5) node (33) [label=below:$X$]{};
                            
    \draw[blue] (34) -- (37);
    \draw[blue] (33) -- (34);
    \draw[blue] (34) -- (36);
    \tikzstyle{every node}=[draw=none,circle,fill=none,minimum size=0pt,
                            inner sep=0pt]

    \draw (3,-1) node (T1) {\Large$T_2$};

\end{tikzpicture}	
	\end{center}
Since the rotation between $T_1$ and $T$ moves the subtrees labeled by $X$ and $Y$ off the right arm, they must remain in the same position relative to one another in $T_2$. 

Suppose $w \in \PW(T_1,T_2)$. Since $T_2$ parses $w$, one must have $w(X) \ne w(Y)$ and, hence, by Proposition~\ref{unwritten}, $T$ parses $w$. Thus $\PW(T_1,T_2) \subset \PW(T_1,T,T_2)$. By definition, $\PW(T_1,T,T_2)\subset \PW(T_1,T_2)$, so $\PW(T_1,T_2)=\PW(T_1,T,T_2)$, as desired.
\end{proof}

\begin{thm} \label{thm: comparable PW}
	Suppose $T<T'$ in $\mathscr{C}_n$ and $\rk(T')-\rk(T)=k$. Then $|\PW(T,T')|=3\cdot 2^{n-1-k}$.
\end{thm}

\begin{proof}
	One proceeds by induction. Proposition~\ref{prop: cover parse word} addresses the case $k=1$. Via repeated leaf reductions, one may assume $T=\RCT{n}$. Now suppose the statement holds for $k-1$, that $T<T'$ and $\rk (T')-\rk(T)=k$. One has a chain in $\mathscr{C}_n$, $T\lessdot T_1 \lessdot T_2 \lessdot \cdots \lessdot T_{k-1} \lessdot T'$. By induction $|\PW(T,T_{k-1})|=3\cdot 2^{n-k}$. One constructs a bijection
	\[
	\PW(T,T_{k-1},T')\to \PW(T,T_{k-1})\smallsetminus \PW(T,T_{k-1},T').		\]

	First, one characterizes those parsewords in $\PW(T,T_{k-1},T')$. One knows that $T_{k-1}$ and $T'$ differ by a right arm rotation.
	\begin{center}
	\begin{tikzpicture}
    \tikzstyle{every node}=[draw,circle,fill=white,minimum size=10pt,
                            inner sep=0pt]

    \draw (-4,1) node (02) [label=below:$X$]{};
    \draw (-3,0) node (03) [label=below:$Y$]{};
    \draw (-1,0) node (04) [label=below:$Z$]{};
    \draw (-4,3) node (07) [label=above:$S$]{};
    
    \tikzstyle{every node}=[draw,circle,fill=black,minimum size=4pt,
                            inner sep=0pt]

    \draw (-3,2) node (05) {};
    \draw (-2,1) node (06) {};
    \draw (05) -- (02);
    \draw (05) -- (06);
    \draw (06) -- (03);
    \draw (06) -- (04);
    \draw (05) -- (07);

    \tikzstyle{every node}=[draw,circle,fill=white,minimum size=10pt,
                            inner sep=0pt]

    \draw (1,0) node (12) [label=below:$X$]{};
    \draw (3,0) node (13) [label=below:$Y$]{};
    \draw (4,1) node (14) [label=below:$Z$]{};
    \draw (2,3) node (17) [label=above:$S$]{};
    
    \tikzstyle{every node}=[draw,circle,fill=black,minimum size=4pt,
                            inner sep=0pt]

    \draw (3,2) node (15) {};
    \draw (2,1) node (16) {};
    \draw (15) -- (14);
    \draw (15) -- (16);
    \draw (16) -- (12);
    \draw (16) -- (13);
    \draw (15) -- (17);
    
    \path[draw,thick,->] (-1,1.5) -- (1,1.5);
    
    \tikzstyle{every node}=[draw=none,circle,fill=none,minimum size=0pt,
                            inner sep=0pt]

    \draw (-4.2,-0.5) node (T1) {\Large$T_{k - 1}$};
    \draw (4,-0.5) node (T2) {\Large$T'$};

\end{tikzpicture}	
	\end{center}
	From Proposition~\ref{unwritten}, one has that $T'$ also parses $w \in \PW(T,T_{k-1})$ (i.e.\ $w \in \PW(T,T_{k-1},T')$) if and only if $w(X) \ne w(Y)$.

	Now define the map
	\[
	\phi \colon \PW(T,T_{k-1})\to \PW(T,T_{k-1})
	\]
	as follows. Suppose $w\in \PW(T,T_{k-1})$. Then $w=w_Sw_Xw_Yw_Z$, where $w_J$ is the word parsed by the leaves of the subtree $J$. Define a transposition in $\mathfrak{S}_{\{0,1,2\}}$ by $\sigma=(w(Y),w(Z))$. Then define $\phi(w)=w_Sw_X\sigma(w_Y)\sigma(w_Z)$. One needs to show that $\phi(w) \in \PW(T,T_{k-1})$.

	On the one hand, $\sigma$ permutes the alphabet within the smallest subtree of $T_{k-1}$ containing both $Y$ and $Z$, while leaving the label of its root unchanged, so $\phi(w)$ is certainly parsed by $T_{k-1}$. Recall that $T$ was assumed to be $\RCT{n}$. Labeling $T$ with $w$ gives
	\begin{center}
		\begin{tikzpicture}[scale=0.4]
%\node (a) at (0,0) [fill,circle,inner sep=1.2pt] {};
\foreach \x in {0,1,2}
{
	\node at (\x-1,-1-\x) [fill,circle,inner sep=1.2pt] {};
	\node at (\x,-\x) [fill,circle,inner sep=1.2pt] {};
}
\foreach \y in {0,1,2}
{
	\draw (\y-1,-1-\y)--(\y,-\y);
}
\draw (0,0)--(1,-1);
\draw (1,-1)--(2,-2);

\foreach \z in {4,5,6,7}
{
	\node at (\z-1,-1-\z) [fill,circle,inner sep=1.2pt] {};
	\node at (\z,-\z) [fill,circle,inner sep=1.2pt] {};
}
\foreach \w in {4,5,6,7}
{
	\draw (\w-1,-1-\w)--(\w,-\w);
}
\foreach \a in {4,5,6}
{
	\draw (\a,-\a)--(\a+1,-\a-1);
}
\draw[dotted] (2,-2)--(4,-4);
%\draw[dotted] (1.5,-2.5)--(3.5,-4.5);

\foreach \b in {5,10}
{
\foreach \z in {4,5,6,7}
{
	\node at (\z+\b-1,-1-\z-\b) [fill,circle,inner sep=1.2pt] {};
	\node at (\z+\b,-\z-\b) [fill,circle,inner sep=1.2pt] {};
}
\foreach \w in {4,5,6,7}
{
	\draw (\w+\b-1,-1-\w-\b)--(\w+\b,-\w-\b);
}
\foreach \a in {4,5,6}
{
	\draw (\a+\b,-\a-\b)--(\a+\b+1,-\a-\b-1);
}
}
\draw[dotted] (7,-7)--(9,-9);
\draw[dotted] (12,-12)--(14,-14);
\draw [decorate,decoration={brace,amplitude=5pt},xshift=-4pt,yshift=-9pt]
   (3,-5)  -- (-1,-1) 
   node [black,midway,below=4pt,xshift=-2pt] {\footnotesize $w_S$};
 \draw [decorate,decoration={brace,amplitude=5pt},xshift=-4pt,yshift=-9pt]
   (8,-10)  -- (4,-6) 
   node [black,midway,below=4pt,xshift=-2pt] {\footnotesize $w_X$};
\draw [decorate,decoration={brace,amplitude=5pt},xshift=-4pt,yshift=-9pt]
   (13,-15)  -- (9,-11) 
   node [black,midway,below=4pt,xshift=-2pt] {\footnotesize $w_Y$};
   \draw [decorate,decoration={brace,amplitude=5pt},xshift=-4pt,yshift=-9pt]
   (19,-21)  -- (14,-16) 
   node [black,midway,below=4pt,xshift=-2pt] {\footnotesize $w_Z$};

 \draw[dotted] (17,-17)--(19,-19);
\node at (19,-19) [fill,circle,inner sep=1.2pt] {};
\node at (18,-20) [fill,circle,inner sep=1.2pt] {};
\node at (20,-20) [fill,circle,inner sep=1.2pt] {};
\draw (18,-20)--(19,-19);
\draw (20,-20)--(19,-19);
\end{tikzpicture}
	\end{center}
Proposition~\ref{prop: common root} means that when the subtree of $T$ containing the leaves of $T_{k-1}$'s $Y$ and $Z$ subtrees is fully labeled, the root of this subtree receives the same label as the root of the smallest subtree of $T_{k-1}$ containing both $Y$ and $Z$, call it $w(YZ)$ and another right arm vertex receives the label $w(Z)$:
%Fully labeling the subtree of $T$ containing the leaves of $T_{k-1}$'s $Y$ and $Z$ subtrees entails giving the root the label $w(Y)$ and another right arm vertex the label $w(Z)$:
	\begin{center}
		\begin{tikzpicture}[scale=0.4]
%\node (a) at (0,0) [fill,circle,inner sep=1.2pt] {};
\foreach \x in {0,1,2}
{
	\node at (\x-1,-1-\x) [fill,circle,inner sep=1.2pt] {};
	\node at (\x,-\x) [fill,circle,inner sep=1.2pt] {};
}
\foreach \y in {0,1,2}
{
	\draw (\y-1,-1-\y)--(\y,-\y);
}
\draw (0,0)--(1,-1);
\draw (1,-1)--(2,-2);

\foreach \z in {4,5,6,7}
{
	\node at (\z-1,-1-\z) [fill,circle,inner sep=1.2pt] {};
	\node at (\z,-\z) [fill,circle,inner sep=1.2pt] {};
}
\foreach \w in {4,5,6,7}
{
	\draw (\w-1,-1-\w)--(\w,-\w);
}
\foreach \a in {4,5,6}
{
	\draw (\a,-\a)--(\a+1,-\a-1);
}
\draw[dotted] (2,-2)--(4,-4);
%\draw[dotted] (1.5,-2.5)--(3.5,-4.5);

\node at (8,-10) [fill,circle,inner sep=1.2pt] {};
\node at (9,-9) [fill,circle,inner sep=1.2pt] {};
\node at (10,-10) [fill,circle,inner sep=1.2pt] {};
\draw (8,-10)--(9,-9);
\draw (9,-9)--(10,-10);
\draw [dotted] (7,-7)--(9,-9);

\draw [decorate,decoration={brace,amplitude=5pt},xshift=-4pt,yshift=-9pt]
   (3,-5)  -- (-1,-1) 
   node [black,midway,below=4pt,xshift=-2pt] {\footnotesize $w_Y$};
 \draw [decorate,decoration={brace,amplitude=5pt},xshift=-4pt,yshift=-9pt]
   (9,-11)  -- (4,-6) 
   node [black,midway,below=4pt,xshift=-2pt] {\footnotesize $w_Z$};

   \draw (0,0) node [above] {\footnotesize{$w(YZ)$}};
   \draw (5,-5) node [above right] {\footnotesize{$w(Z)$}};
\end{tikzpicture}
	\end{center}

Since $w(YZ)$ is equal to neither $w(Y)$ nor $w(Z)$, it is fixed by $\sigma$. Consequently, applying $\sigma$ to the subtree of $T$ consisting of those vertices in $Y$ and $Z$ has the same effect on parsewords as applying $\sigma$ to $Y$ and $Z$ in $T_{k-1}$. In other words, $\phi(w)$ is parsed by $T$, so the map is well-defined.
%Then applying $\sigma$ to the vertices of $T$ in the $Y$ and $Z$ subtrees of $T_{k-1}$ separately is the same as applying $\sigma$ to both the $Y$ and $Z$ parts of $T$ simultaneously. Thus $\phi(w)$ parses $T$, so the map is well-defined.

Then $\phi$ is transparently a bijection 
\[
	\PW(T,T_{k-1})\to \PW(T,T_{k-1})
\]
and, moreover, exchanges $\PW(T,T_{k-1},T')$ and $\PW(T,T_{k-1})\smallsetminus \PW(T,T_{k-1},T')$: if $w \in \PW(T,T_{k-1},T')$, then $\phi(w)$ has $\phi(w)(X)=\phi(w)(Y)$, meaning it cannot parse $T'$. Thus, $\phi$ is a bijection
	\[
	\PW(T,T_{k-1},T')\to \PW(T,T_{k-1},T')\smallsetminus \PW(T,T_{k-1},T').		\]
	Consequently, $\PW(T,T_{k-1},T')$ contains precisely half the parsewords of $\PW(T,T_{k-1})$, i.e.\ there are $3\cdot 2^{n-1-k}$ of them. From Theorem~\ref{thm: PWchain}, one has $\PW(T,T_{k-1},T')=\PW(T,T')$, so $|\PW(T,T')|=3\cdot 2^{n-1-k}$, as desired.
\end{proof}

\begin{cor}\label{cor: commonincomb}
Since $k \leq n - 2$ by Proposition \ref{prop: maxrank}, any pair of trees comparable in $\mathscr{C}_n$ has a common parse word.
\end{cor}

\begin{thm} \label{thm: join-meet}
Suppose $T_1$ and $T_2$ have an upper bound in $\mathscr{C}_n$. Then, $$\PW(T_1,T_2)=\PW(T_1\wedge T_2,T_1\vee T_2).$$
\end{thm}

\begin{proof}
  The statement is clear when $T_1$ and $T_2$ are comparable, so assume $T_1$ and $T_2$ are not comparable. 
  By Theorem~\ref{thm: PWchain}, 
      \[
  	\PW(T_1\wedge T_2,T_1\vee T_2,T_1)=\PW(T_1\wedge T_2,T_1\vee T_1)=\PW(T_1\wedge T_2,T_1\vee T_2,T_2).
  	\]
	One then immediately has that $\PW(T_1\wedge T_2,T_1 \vee T_2)\subset \PW(T_1,T_2)$.
	
	All that remains is to show inclusion the other way. Suppose the theorem holds for trees with $k<n$ leaves. Without loss of generality, one may assume that $T_1 \wedge T_2=\RCT{n}$ by making repeated leaf reductions in $T_1$ and $T_2$.  There are two cases:
	\begin{enumerate}
		\item Suppose $T_1$ and $T_2$ share a leaf reduction at, say, $i$. Then $T_1\wedge T_2$ and $T_1 \vee T_2$ must also share this leaf reduction. Then,
		  \begin{align*}
		  \PW(T_1\wedge T_2,T_1\vee T_2)&=f^<_i(\PW(\widehat{T_1\wedge T_2},\widehat{T_1\vee T_2}))\\
		  &\phantom{=}\sqcup f^>_i(\PW(\widehat{T_1\wedge T_2},\widehat{T_1\vee T_2}))\\
		  &=f^<_i(\PW(\hat{T_1},\hat{T_2}))\sqcup f^>_i(\PW(\hat{T_1},\hat{T_2}))\\
		  &=\PW(T_1,T_2),
		\end{align*}
		  as desired, with the inductive hypothesis being used in the second equality.
		\item Suppose, on the other hand, that no such common leaf reduction exists.
		Then one must have that $T_1 \wedge T_2 = \RCT{n}$. Since $T_1 \vee T_2$ exists, one must then have that all parenthesis pairs in $RP_{T_1}$ and $RP_{T_2}$ are disjoint. Suppose $w=w_1w_2\cdots w_n \in \PW(T_1,T_2)$. Then, since $T_1 \meet T_2$ is $\RCT{n}$, either $RP_{T_1}$ or $RP_{T_2}$ contains a parenthesis pair enclosing $a_{n-1}$, else both trees would have a leaf reduction at $(n-1,n)$. Without loss of generality, assume $RP_{T_1}$ contains a parenthesis pair enclosing $a_{n-1}$. Moreover, $RP_{T_1}$ has a maximal parenthesis pair enclosing the leaves $a_j,\ldots,a_{n-1}$. 
		  Then, since all parenthesis pairs in $RP_{T_1}$ and $RP_{T_2}$ are disjoint, none of $a_j,\ldots, a_{n-1}$ are enclosed by a parenthesis pair in $RP_{T_2}$. Consequently, the subtrees of $T_2$ and $T_1 \meet T_2$ with leaf set $a_j,\ldots,a_n$ are both isomorphic to $\RCT{n-j+1}$. Call this subtree $X_1$.
		  By the maximality of the parenthesis pair containing $a_j,\ldots,a_{n-1}$, one has that $T_1$ and $T_1 \join T_2$ have isomorphic subtrees whose leaf sets are $a_j,\ldots,a_{n}$, call this subtree $X_2$. Consequently, $w_j\cdots w_n$ is parsed by the subtree containing $a_j,\ldots,a_n$ in $T_1,T_2,T_1\wedge T_2$ and $T_1\vee T_2$. Then, from Proposition~\ref{prop: common root}, one has that $w(X_1)=w(X_2)$. Collapse the subtrees $X_1$ and $X_2$ to obtain $T_1',T_2', T_1'\wedge T_2'$ and $T_1'\vee  T_2'$. By induction, $\PW(T_1',T_2')=\PW(T_1'\wedge T_2',T_1'\vee T_2')$, meaning $w_1\cdots w_{j-1}w({X_1})=w_1\cdots w_{j-1}w({X_2})$ is parsed by $T_1'\meet T_2'$ and $T_1' \join T_2'$. It is then easy to see that this implies $w$ lies in $\PW(T_1\wedge T_2, T_1 \vee T_2)$, as desired.
	\end{enumerate}
\end{proof}

\begin{remark}
  If $T_1$ and $T_2$ have an upper bound in $\mathscr{C}_n$, and $\rk(T_1\join T_2)-\rk(T_1\meet T_2)=k$, combining Theorem~\ref{thm: join-meet} and Theorem~\ref{thm: comparable PW}, one has
  \[
	|\PW(T_1,T_2)|=|\PW(T_1\meet T_2,T_1\join T_2)|=3\cdot 2^{n-1-k}.
  \]
\end{remark}

\section*{Acknowledgements}
This research was carried out at the School of Mathematics, University of Minnesota - Twin Cities, under the supervision of Vic Reiner and Dennis Stanton with funding from NSF grant DMS-1001933. We would like to thank Vic Reiner and Dennis Stanton for their invaluable guidance and support, Bobbe Cooper for giving a talk introducing us to this problem, Nathan Williams for his words of advice and continual assessment of our progress, as well as Alan Guo and an anonymous referee for their input in the proof of Theorem \ref{thm: |rank|}. We would also like to thank the referee for their very thoughtful and detailed comments.

\bibliography{combtamari}
\nocite{friedmantamari}
\bibliographystyle{habbrv}
\end{document}